\documentclass[11pt,reqno]{amsart}
\usepackage{graphicx,color}
\usepackage{amssymb,amscd,latexsym, mathrsfs, epsfig}
\usepackage[all]{xy}
\usepackage[latin1]{inputenc}
\usepackage{verbatim}
\usepackage{youngtab}
\usepackage{enumitem}
\usepackage{todonotes}
\usepackage{calrsfs}
\newtheorem{thm}{Theorem}[section]
\newtheorem{lem}[thm]{Lemma}
\newtheorem{pro}[thm]{Proposition}
\newtheorem{que}[thm]{Question}

\newtheorem{kor}[thm]{Corollary}
\newtheorem{defi}[thm]{Definition}
\newtheorem{conj}[thm]{Conjecture}

\theoremstyle{remark}

\newtheorem{exam}[thm]{Example}
\newtheorem{rem}[thm]{Remark}

\newcommand{\un}[1]{\underline{#1}}

\newcommand{\blank}{\underline{\hspace{0.3cm}}}
\newcommand{\Fl}{\mathrm{Fl}}
\newcommand{\filt}[6]{
\[
\begin{xy}
\xymatrix@R20pt@C20pt{
&\mathbb{C}^3&\\\langle #1,#2 \rangle\ar[ru]&\langle #3,#4\rangle\ar[u]&\langle #5,#6\rangle\ar[lu]\\\langle #1\rangle\ar[u]&\langle #3\rangle\ar[u]&\langle #5\rangle\ar[u]}
\end{xy}
\]
}
\newcommand{\quiverext}[6]{\[
\begin{xy}
\xymatrix@R20pt@C20pt{
#1&&#2\ar[ld]^{#5}\ar[ll]^{#4}\\&#3\ar[lu]^{#6}}
\end{xy}
\]
}
\newcommand{\quiverhom}[6]{\[
\begin{xy}
\xymatrix@R20pt@C20pt{
(#1)&&(#2)\ar@{--}[ld]^{#5}\ar[ll]^{#4}\\&(#3)\ar[lu]^{#6}}
\end{xy}
\]
}

\newcommand\Qn{\mathbb{Q}}
\newcommand\Zn{\mathbb{Z}}
\newcommand\Nn{\mathbb{N}}
\newcommand\Cn{\mathbb{C}}

\newcommand\udim{\underline{\mathrm{dim}}\,}
\newcommand{\Gr}{\mathrm{Gr}}
\newcommand{\Sc}[2]{\langle #1,#2\rangle}
\newcommand{\Hom}{\mathrm{Hom}} 
\newcommand{\Ext}{\mathrm{Ext}} 
\newcommand{\ext}{\mathrm{ext}} 
\newcommand{\End}{\mathrm{End}} 
\newcommand{\Rep}{\mathrm{Rep}} 
\newcommand{\ses}[3]{0\rightarrow #1\rightarrow #2\rightarrow#3\rightarrow 0}

\DeclareMathOperator{\Spec}{Spec}
\DeclareMathOperator{\pr}{pr}
\DeclareMathOperator{\rk}{rk}
\DeclareMathOperator{\ch}{ch}

\newcommand{\EE}{\mathcal{E}}
\newcommand{\UU}{\mathcal{U}}
\newcommand{\QQ}{\mathcal{Q}}
\newcommand{\sub}{\subseteq}
\newcommand{\xto}[2]{\xrightarrow[#1]{#2}}

\newcommand{\xot}[2]{\xleftarrow[#1]{#2}}
\renewcommand{\P}{\mathbb{P}}

\DeclareSymbolFont{symbolsC}{U}{pxsyc}{m}{n}
\DeclareMathSymbol{\Perp}{\mathrel}{symbolsC}{121}

\begin{document}
\title[Non-Schurian indecomposables via intersection theory]{Non-Schurian indecomposables via intersection theory}
\author{Hans Franzen}
\address{Hans Franzen: Math.\ Institut der Universit\"at Bonn, Endenicher Allee 60, 53115 Bonn, Germany}
\email{franzen@math.uni-bonn.de}
\author{Thorsten Weist}
\address{Thorsten Weist: Bergische Universit\"at Wuppertal, Gau\ss str.\ 20, 42097 Wuppertal, Germany}
\email{weist@uni-wuppertal.de}
%
\presetkeys{todonotes}{size=\scriptsize}{}

\begin{abstract}
For an acyclic quiver with three vertices, we consider the canonical decomposition of a non-Schurian root and associate certain representations of a generalized Kronecker quiver. These representations correspond to points contained in the intersection of two subvarieties of a Grassmannian and give rise to representations of the original quiver, preserving indecomposability. We show that these subvarieties intersect using Schubert calculus. Provided that it contains a Schurian representation, the dimension of the intersection is what we expect by Kac's Theorem.
\end{abstract}

\maketitle
\section{Introduction}
\noindent The classification of indecomposable quiver representations is a very hard problem in general. In many cases, it is not known how to construct just one indecomposable representation of a given dimension explicitly. By a Theorem of Kac \cite[Theorem C]{kac2}, the number of parameters of isomorphism classes of indecomposables of a fixed dimension is determined by the Euler form of the quiver. We combine representation-theoretic and geometric methods in order to describe such a parameter family of indecomposable representations for certain non-Schurian roots of an acyclic quiver with three vertices which is fixed from now on. Actually, the results can be extended straightforwardly to non-Schurian roots whose canonical decomposition consists of an exceptional root and an imaginary Schur root which corresponds to a root of a generalized Kronecker quiver. 

Given a Schur root $\alpha$, the methods of \cite{wei} can be used to construct $(1-\langle \alpha, \alpha \rangle)$-parametric families of indecomposable representations, as predicted by Kac's Theorem. The main reason why there is no natural generalization to the non-Schurian case is that the canonical decomposition of a non-Schurian root does not consist of the root itself while the canonical decomposition of a Schur root does. This ensures that glueing representations recursively, following the algorithm of \cite{dw}, ends with indecomposable representations. The canonical decomposition of a non-Schurian root of a quiver with three vertices consists of an imaginary Schur root and an exceptional  root where the imaginary root corresponds to one of a generalized Kronecker quiver. This suggests to try to construct indecomposable representations using Ringel's reflection functor introduced in \cite{rin2}. In general, it is not clear how to characterize representations which can be obtained in this way or how to decide whether there even exist any indecomposables of this shape. In this paper, we introduce a  method which enables us to answer this question geometrically. 

For a fixed non-Schurian root $\alpha$ of a quiver with three vertices, we assume that its canonical decomposition reads as $\alpha = \alpha_1^{d_1} \oplus \hat{\alpha}$, where $\alpha_1$ is a real Schur root and $\hat{\alpha} \in \alpha_1^\perp$ is imaginary (the other possible case for the canonical decomposition is dual; see Lemma \ref{l:can_decomp}). A representation $M_{\hat{\alpha}} \in \alpha_1^\perp$ (the subscript indicates that its dimension vector is $\hat{\alpha}$) and a $d_1$-dimensional subspace $U$ of $\Ext(M_{\hat{\alpha}},S_{\alpha_1})$ give rise to a short exact sequence
$$
	0 \to S_{\alpha_1}^{d_1} \to M_{\alpha} \to M_{\hat{\alpha}} \to 0
$$
such that $M_\alpha$ is indecomposable whenever $M_{\hat{\alpha}}$ is (cf.\ Proposition \ref{p:ses_indec}). In this context $S_{\alpha_1}$ is the indecomposable of dimension vector $\alpha_1$. The construction of a $(1 - \langle \alpha,\alpha \rangle)$-family of indecomposables $M_\alpha$ thus reduces to finding a respective family of indecomposables $M_{\hat{\alpha}}$ in $\alpha_1^\perp$ with $\dim\Ext(M_{\hat\alpha},S_{\alpha_1})\geq d_1$.


Let $\alpha_2$ and $\alpha_3$ be the two simple roots of the category $\alpha_1^\perp$. Then $\alpha_1^\perp$ is equivalent to the category of a generalized Kronecker quiver. Writing $\hat{\alpha} = \smash{\alpha_2^{d_2}} + \smash{\alpha_3^{d_3}}$, the root $\hat{\alpha}$ corresponds to the root $(d_3,d_2)$ of $K(\ext(\alpha_3,\alpha_2))$  if we assume that $\ext(\alpha_2,\alpha_3)=0$. Consider the short exact sequence $0 \to \smash{S_{\alpha_2}^{\ext(\alpha_3,\alpha_2)}} \to S_\delta \to S_{\alpha_3} \to 0$ induced by the choice of a basis of $\Ext(S_{\alpha_3},S_{\alpha_2})$. Set $r = \ext(\alpha_3,\alpha_2)d_3-d_2$. Suppose we had a commutative diagram with exact rows
\begin{equation*}
	\begin{xy} \xymatrix@R20pt@C25pt{0\ar[r]&S_{\alpha_2}^r\ar[r]\ar[d]^{f_1}&S_{\delta}^{d_3}\ar[d]^{f_2}\ar[r]&M_{\hat\alpha}\ar[r]\ar[d]^{f_3}&0\\0\ar[r]&S_{\alpha_1}^t\ar[r]&S_{\alpha_1}^s\ar[r]&S_{\alpha_1}^{d_1+\Sc{\hat\alpha}{\alpha_1}}\ar[r]&0}
	\end{xy}
\end{equation*}
for certain $s,t\geq 0$ such that all $f_i$'s are of maximal rank. This would imply $\dim\Ext(M_{\hat\alpha},S_{\alpha_1})\geq d_1$. If $d_3 = 1$ then $M_{\hat{\alpha}}$ would even be indecomposable. If $d_3 \geq 2$ then indecomposability is not guaranteed but in many cases a recursive procedure applies to construct $M_{\hat{\alpha}}$ from smaller dimensions, assuring that it is indecomposable. 

We investigate under which circumstances a diagram as above can be found. For these purposes, we restrict to roots $\alpha$ that we call of type one; a mild numerical condition that we define in Definition \ref{typeone}. We consider the Grassmannian $\Gr_r(\Hom(S_{\alpha_2},S_{\alpha_1}^s))$ -- its points can be regarded as injective morphisms $S_{\alpha_2}^r \to S_{\alpha_1}^s$ -- and define subvarieties $X_1^\alpha$ and $X_2^\alpha$ inside this Grassmannian such that a point lies in $X_1^\alpha$ (resp.\ $X_2^\alpha$) if and only if the corresponding morphism factors through $S_{\alpha_1}^t$ (resp.\ \smash{$S_{\delta}^{d_3}$}). Denote $I^\alpha = X_1^\alpha \cap X_2^\alpha$. Our main result states:

\begin{thm}\label{mainresults}
Let $\alpha$ be a non-Schurian root of type one. Then we have:
\begin{enumerate}
\item A point of $I^\alpha$ corresponds to a representation $M_{\hat{\alpha}} \in \alpha_1^\perp$ for which $\dim \Ext(M_{\hat{\alpha}},S_{\alpha_1}) \geq d_1$. Two such representations are isomorphic if and only if the corresponding points lie in the same $\mathrm{Gl}_{d_3}(k)$-orbit.
\item Every pair $(M_{\hat\alpha},e) \in I^\alpha \times \Gr_{d_1}(\Ext(M_{\hat{\alpha}},S_{\alpha_1}))$ gives rise to a short exact sequence $\ses{S_{\alpha_1}^{d_1}}{M_{\alpha}}{M_{\hat\alpha}}$ such that $M_{\alpha}$ is indecomposable if and only if $M_{\hat\alpha}$ is indecomposable. 

Given another pair $(M_{\hat{\alpha}}',e') \in I^\alpha\times \Gr_{d_1}(\Ext(M_{\hat\alpha}',S_{\alpha_1}))$ with short exact sequence $\ses{S_{\alpha_1}^{d_1}}{M'_{\alpha}}{M'_{\hat\alpha}}$, the middle terms
$M_\alpha$ and $M_\alpha'$ are isomorphic if and only if $e$ and $e'$ are isomorphic.
\item The closed subset $I^{\alpha}$ is non-empty and each of its irreducible components has dimension at least $d_3^2-\Sc{\alpha}{\alpha}$.
\item If $I^\alpha$ contains one Schurian representation $M_{\hat\alpha}$ then the corresponding irreducible component is of dimension $d_3^2-\Sc{\alpha}{\alpha}$ and contains an open subset of Schurian representations. In this case, there exists a $(1-\Sc{\alpha}{\alpha})$-parameter family of isomorphism classes of Schurian representations of dimension $\hat\alpha$ with $M_{\hat\alpha}\in S_{\alpha_1}^{\perp}$ and $\dim\Ext(M_{\hat\alpha},S_{\alpha_1})\geq d_1$.
\item If there exist a decomposition $\alpha=\beta+\gamma$, which is compatible with the canonical decomposition of $\alpha$, and Schur representations $M_{\hat\beta}\in I^\beta$ and $M_{\hat\gamma}\in I^\gamma$ with $\Hom(M_{\hat\beta},M_{\hat\gamma})=\Hom(M_{\hat\gamma},M_{\hat\beta})=0$, then there is a non-empty open subset of Schurian representations in $I^\alpha$ which can be constructed recursively from Schurian representations of dimensions $\hat\beta$ and $\hat\gamma$.
\end{enumerate}
\end{thm}

The first assertion of the theorem follows from Theorem \ref{intersection}, the second from Proposition \ref{p:ses_indec}. Having determined the dimensions of both $X_1^\alpha$ and $X_2^\alpha$ it is immediate that every irreducible component of $I^\alpha$ is at least $(d_3^2 - \langle \alpha,\alpha \rangle)$-dimensional -- provided that $I^\alpha$ is non-empty; this is Theorem \ref{lem_dim_int}. We use Schubert calculus to show that $I^\alpha \neq \emptyset$. More precisely we display the class $[X_1] \cdot [X_2]$ in the Chow ring of the ambient Grassmannian as a (positive) linear combination $\sum_\lambda d_\lambda \Delta_\lambda$ (see Theorem \ref{thm_inter_prod}) and exhibit one partition $\lambda$ for which we can read off $d_\lambda \neq 0$.
As there exists a $\mathrm{Gl}_{d_3}(k)$-action on $I^\alpha$ whose orbits are precisely the isomorphism classes, Kac's theorem yields the upper bound for the dimension of $I^\alpha$ indicated in the fourth part, see Remark \ref{kacgen}. Finally, it can be shown that the map $\dim\Hom$ is upper semi-continuous on $I^\beta\times I^\gamma$ when $\alpha=\beta+\gamma$ is compatible with the canonical decomposition of $\alpha$. Combining this observation with the methods of \cite{wei3}, the last part of the theorem is obtained in Theorems \ref{glueing1} and \ref{glueing2}.

We discuss special cases in which the mentioned numerical conditions are violated in subsection \ref{special} and conclude with subsection \ref{exs} where we illustrate our methods with several examples.

\subsection*{Acknowledgements} The authors are grateful to Michael Ehrig for explaining the proof of Lemma \ref{lem_Michael}. Moreover, we would like to thank Andrew Hubery and Markus Reineke for helpful comments.

\section{Preliminaries}\label{allg}
\subsection{Quiver representations}\label{rr}
\noindent We recall the basic notions concerning the representation theory of quivers. For a detailed introduction, we refer to \cite{ass}. Let $k$ be an algebraically closed field of characteristic $0$. Let $Q=(Q_0,Q_1)$ be a quiver with vertices $Q_0$ and arrows $Q_1$ denoted by $\rho:i\rightarrow j$ for $i,j\in Q_0$. Let $s:Q_1\to Q_0$ and $t:Q_1\to Q_0$ be defined by mapping an arrow to its source and target respectively. Throughout the paper we assume that $Q$ is connected. For the remaining part of this section, we assume $Q$ to have no oriented cycles. Define the abelian group
\[\mathbb{Z}Q_0=\bigoplus_{i\in Q_0}\mathbb{Z}i\] and its monoid of dimension vectors $\mathbb{N}Q_0$. If
\[M=((M_i)_{i\in Q_0},(M_{\rho}:M_{s(\rho)}\rightarrow M_{t(\rho)})_{\rho\in Q_1})\]
 is a finite-dimensional $k$-representation of $Q$, the dimension vector $\udim M\in\mathbb{N}Q_0$ of $M$ is defined by
$\udim M=\sum_{i\in Q_0}(\dim_kM_i)i.$ We denote by $\mathrm{Rep}(Q)$ the category of finite-dimensional representations of $Q$. 

For a fixed dimension vector $\alpha\in\mathbb{N}Q_0$, the variety $R_\alpha(Q)$ of $k$-representations of $Q$ of dimension vector
$\alpha$ is defined as the affine $k$-space
\[R_{\alpha}(Q)=\bigoplus_{\rho\in Q_1} \mathrm{Hom}_k(k^{\alpha_{s(\rho)}},k^{\alpha_{t(\rho)}}).\]

We will frequently use the notation $M_{\alpha}$ which just means that $M_\alpha$ is an object of $\Rep(Q)$ of dimension $\alpha$. If a property is independent of the point chosen in some non-empty open subset $U$ of $R_{\alpha}(Q)$, following \cite{sch}, we say that this property is true for a general representation of dimension $\alpha$.

On $\Zn Q_0$ we have a non-symmetric bilinear form, the Euler form,
which is defined by
\[\Sc{\alpha}{\beta}=\sum_{i\in Q_0}\alpha_i\beta_i-\sum_{\rho\in Q_1}\alpha_{s(\rho)}\beta_{t(\rho)}\]
for $\alpha,\,\beta\in\Zn Q_0$.
A dimension vector is called a root if there exists an indecomposable representation of this dimension. It is called Schur root if there exists a Schurian representation of this dimension, i.e. with trivial endomorphism ring. 
A root $\alpha\in\Nn Q_0$ is called real if we have $\alpha\in W(Q)Q_0$, i.e. $\alpha$ arises by reflecting a simple root, where $W(Q)$ denotes the Weyl group of the quiver. All the other roots are called imaginary. Recall that a root is real if and only if $\Sc{\alpha}{\alpha}=1$ and imaginary if and only if $\Sc{\alpha}{\alpha}\leq 0$. If $\alpha$ is real, there exists a unique indecomposable representation up to isomorphism which we denote by $S_\alpha$.

Let $(\alpha,\beta):=\Sc{\alpha}{\beta}+\Sc{\beta}{\alpha}$ be the symmetrized Euler form. The fundamental domain $F(Q)$ of $\mathbb{N}Q_0$ is given by the dimension vectors $\alpha$ with connected support such that $(\alpha,i)\leq 0$ for all $i\in Q_0$. We have $\alpha\in W(Q)F(Q)$ for all imaginary roots $\alpha$.

By \cite{rin2}, for two representations $M$, $N$ of $Q$ we have
\[\Sc{\udim  M}{\udim  N}=\dim_k\Hom(M,N)-\dim_k\Ext(M,N)\]
and $\Ext^i(M,N)=0$ for $i\geq 2$.

Since the function
$\lambda:R_{\alpha}(Q)\times R_{\beta}(Q)\rightarrow\mathbb{N},\,(M,N)\mapsto\dim\Hom(M,N)$, is upper semi-continuous, see for instance \cite{sch}, we can define $\hom(\alpha,\beta)$ to be the minimal, and therefore general, value of this function. In particular, we get that if $\alpha$ is a Schur root of a quiver, then a general representation is Schurian. Moreover, let $\ext(\alpha,\beta):=\hom(\alpha,\beta)-\Sc{\alpha}{\beta}$.
We write $\beta\hookrightarrow \alpha$ if a general representation of dimension $\alpha$ has a subrepresentation of dimension $\beta$. As $\ext(\alpha,\alpha)=0$ for real Schur roots, we also call them exceptional roots.

By \cite{kac}, for every dimension vector $\alpha\in\Nn Q_0$, there exists a decomposition $\alpha=\sum_{i\in I}\alpha_i$ such that a general representation $N$ of dimension $\alpha$ is a direct sum of Schurian representations $M_i$ with $\udim  M_i=\alpha_i$. This is called the canonical decomposition of $\alpha$ denoted by $\alpha=\bigoplus_{i\in I}\alpha_i$. We recall the following results:
\vspace{0.3cm}


\begin{thm}[{\cite[Proposition 3]{kac2}, \cite[Theorem 4.4]{sch}}]~
\begin{enumerate}
\item For a general representation $N$ of dimension vector $\alpha$, we have $N\cong\bigoplus_{i\in I} M_i$ with $\udim  M_i=\alpha_i$ if and only if $\ext(\alpha_i,\alpha_j)=0$ for $i\neq j$. Moreover, each $M_i$ is Schurian.
\item Let $\alpha$ be a root. Then up to multiplicity there exists at most one imaginary Schur root in its canonical decomposition.
\end{enumerate}
\end{thm}
Note that \cite[section 4]{dw} gives a very useful algorithm which can be used to determine the canonical decomposition. By \cite[Proposition 7]{dw} we can order the canonical decomposition as 
$$\alpha=\alpha_1^{d_1}\oplus\cdots\oplus\alpha_k^{d_k}$$
where $k\leq |Q_0|$, $\hom(\alpha_i,\alpha_j)=0$ if $i<j$ and $d_i=1$ whenever $\alpha_i$ is imaginary with $\Sc{\alpha_i}{\alpha_i}\neq 0$. We will need the following result of Schofield:
\begin{thm}[{\cite[Theorem 4.1]{sch}}]\label{schofield}
Let $\alpha$ and $\beta$ be Schur roots such that $\mathrm{ext}(\alpha,\beta)=0$. Then we either have $\mathrm{hom}(\beta,\alpha)=0$ or $\mathrm{ext}(\beta,\alpha)=0$. If both $\alpha$ and $\beta$ are imaginary, then we have $\mathrm{hom}(\beta,\alpha)=0$.
\end{thm}
Later we will need the following lemma concerning the canonical decomposition:
\begin{lem}\label{injprojcanon}
Let $\alpha$ be a root with canonical decomposition $\alpha=\bigoplus_{i=1}^k\alpha_i^{d_i}$ such that a general representation of this dimension is neither preprojective nor preinjective. Then for every $i\in \{1,\ldots,k\}$, we have that $\alpha_i$ is neither the dimension vector of a preprojective nor of a preinjective representation. 
\end{lem}
\begin{proof}
Assume that $\alpha$ is a root which does not satisfy the claim, i.e. there exists an $i\in \{1,\ldots,k\}$ such that $\alpha_i$ is either preinjective or preprojective. By applying the reflection functor of \cite{bgp} and passing to the transpose quiver if necessary, we can without loss of generality assume that there exists an $i\in \{1,\ldots,k\}$ such that $\alpha_i$ is injective. 

On the one hand, a general representation of dimension $\alpha$ has a subrepresentation of dimension $\alpha_i$ and we have $\hom(\alpha_i,\alpha)\geq\Sc{\alpha_i}{\alpha}>0$. But on the other hand, for an indecomposable representation $M_{\alpha}$ which is not injective and an injective representation $M_{\alpha_i}$ we have $\Hom(M_{\alpha_i},M_{\alpha})=0$, see for instance \cite[Lemma VIII.2.5]{ass}. Thus we already have $\hom(\alpha_i,\alpha)=0$, which yields a contradiction.
\end{proof}
We will also need the following well-known lemma:
\begin{lem}[{\cite[Lemma 4.1]{hr}}]\label{happelringel}
Let $M$ and $N$ be two indecomposable representations of $Q$ such that we have $\Ext(N,M)=0$. Then every non-zero homomorphism $f:M\rightarrow N$ is either injective or surjective. 
\end{lem}
                                                 
We shortly recall the notion of coefficient quivers following the presentation given in \cite{rin1}. Let $M$ be a representation of dimension $\alpha$. A basis of $M$ is a subset $\mathcal{B}$ of $\bigoplus_{i\in Q_0}M_i$ such that
\[\mathcal{B}_i:=\mathcal{B}\cap M_i\] is a basis of $M_i$ for all vertices $i\in Q_0$. For every arrow $\rho:i\to j$, we may write $M_{\rho}$ as a $(\alpha_{j}\times \alpha_i)$-matrix $M_{\rho,\mathcal{B}}$ with coefficients in $k$ such that the rows and columns are indexed by $\mathcal{B}_{j}$ and $\mathcal{B}_{i}$ respectively. If
\[M_{\rho}(b)=\sum_{b'\in\mathcal{B}_{j}}\lambda_{b',b}b'\]
with $\lambda_{b',b}\in k$ and $b\in\mathcal B_i$, we obviously have $(M_{\rho,\mathcal{B}})_{b',b}=\lambda_{b',b}$.
\begin{defi}
The coefficient quiver $\Gamma(M,\mathcal{B})$ of a representation $M$ with a fixed basis $\mathcal B$ has vertex set $\mathcal{B}$ and arrows between vertices are defined by the condition: if $(M_{\rho,\mathcal{B}})_{b',b}\neq 0$, there exists an arrow $(\rho,b,b'):b\rightarrow b'$.

A representation $M$ is called a tree module if there exists a basis $\mathcal{B}$ for $M$ such that the corresponding coefficient quiver is a tree.

\end{defi}
 In section \ref{glueing}, we need some easy observation concerning coefficient quivers including the following definition.
\begin{defi}
We call a full subquiver $Q'$ of a quiver $Q$ of sink-type if we have for all arrows $\rho\in Q_1$ with $s(\rho)\in Q'_0$ that $t(\rho)\in Q_0'$. Moreover, we call a full subquiver $Q'$ of a quiver $Q$ of source-type if we have for all arrows $\rho\in Q_1$ with $s(\rho)\notin Q'_0$ that $t(\rho)\notin Q_0'$.
\end{defi}

In terms of coefficient quivers, every subquiver of sink-type (resp.\ source-type) defines a subrepresentation (resp.\ factor) in the natural way. Recall that, by \cite{rin1}, every exceptional representation is a tree module. Moreover, for two fixed representations, we can always choose a tree-shaped Ext-basis, see for instance \cite{wei} for more details. This means that a coefficient quiver of the middle term of the exact sequence described by any basis element is obtained by glueing the coefficient quivers via exactly one arrow. The following is straightforward:
\begin{lem}\label{sinksource}
Let $M,N$ be two representations such that $M$ is exceptional with coefficient quiver $\Gamma_M$. Then we have:
\begin{enumerate}
\item If $M^n\subset N$, there exists a coefficient quiver $\Gamma_N$ of $N$ such that $\Gamma_N$ has $n$ subquivers $\Gamma_M^i$ of sink-type where $\Gamma_M^i=\Gamma_M$ for all $i=1,\ldots,n$. 
\item If $N$ has a coefficient quiver $\Gamma_N$ which has $n$ subquivers $\Gamma_M^i$ of sink-type such that $\Gamma_M^i=\Gamma_M$ for all $i=1,\ldots,n$ and such that $\Gamma_M^i\cap\Gamma_M^j=\emptyset$ for $i\neq j$, we have $\dim\Hom(M,N)\geq n$. 
\end{enumerate}
The analogous statements hold when $M^n$ is a factor of $N$.
\end{lem}

\subsection{Exceptional sequences and perpendicular categories}
An indecomposable representation $M$ of a quiver $Q$ is called exceptional if $\Ext(M,M)=0$. With the help of Lemma \ref{happelringel}, it already follows that $\udim  M$ is a real Schur root and $\mathrm{End}(M)=k$. In the following, we do not always distinguish between a real root $\alpha$ and the unique indecomposable representation $S_\alpha$ of this dimension. 

A sequence $S=(M_1,\ldots,M_r)$ of representations of $Q$ is called exceptional if every $M_i$ is exceptional and, moreover, $\Hom(M_i,M_j)=\Ext(M_i,M_j)=0$ if $i<j$. If, in addition, $\Hom(M_j,M_i)=0$ if $i<j$, we call such a sequence reduced. For an exceptional sequence $S=(S_{\alpha_1},\ldots,S_{\alpha_r})$, we denote by $\mathcal{C}(S_{\alpha_1},\ldots,S_{\alpha_r})$ the smallest category containing $S$ and which is closed under extension, kernels of epimorphisms and cokernels of monomorphisms.

For two roots $\beta$ and $\gamma$, we write $\beta\in\gamma^{\perp}$ if $\hom(\gamma,\beta)=\ext(\gamma,\beta)=0$. In this way we can also talk about exceptional sequences of roots.

For a set $S=\{M_1,\ldots,M_r\}$ of representations of $Q$, we define its perpendicular categories
\[^{\perp}S=\{M\in\mathrm{Rep}(Q)\mid \Hom(M,M_j)=\Ext(M,M_j)=0\text{ for }j=1,\ldots,r\},\] 
\[S^{\perp}=\{M\in\mathrm{Rep}(Q)\mid \Hom(M_j,M)=\Ext(M_j,M)=0\text{ for }j=1,\ldots,r\}.\]

It is straightforward to check that these categories are closed under direct sums, direct summands, extensions, images, kernels and cokernels. Schofield proved the following results:
\begin{thm}[{\cite[Theorems 2.3 and 2.4]{sc2}}]\label{perpcat}Let $Q$ be a quiver with $n$ vertices and $S=(\alpha_1,\ldots,\alpha_r)$ be an exceptional sequence. 
\begin{enumerate}
\item The categories $^{\perp}S$ and $S^{\perp}$ are equivalent to the categories of representations of quivers $Q(^{\perp}S)$ and $Q(S^{\perp})$ respectively such that these quivers have $n-r$ vertices and no oriented cycles.
\item There is an isometry with respect to the Euler form between the dimension vectors of  $Q(^{\perp}S)$ (resp. $Q(S^{\perp})$) and the dimension vectors of $^{\perp}S$ (resp. $S^{\perp}$) which is given by
$\Phi((d_1,\ldots,d_{n-r}))=\sum_{i=1}^{n-r}d_i\beta_i$ where $\beta_1,\ldots,\beta_{n-r}$ are the dimension vectors of the simple representations of the perpendicular categories.
\end{enumerate}
\end{thm}

Suppose that $S=(S_{\alpha_1},\ldots,S_{\alpha_r})$ is a reduced exceptional sequence. By \cite[Lemma 2.35]{dw2} we have that $S_{\alpha_1},\ldots,S_{\alpha_r}$ are the simple objects of $\mathcal{C}(S_{\alpha_1},\ldots,S_{\alpha_r})$. Thus, by Theorem \ref{perpcat}, it follows that the category $\mathcal{C}(S_{\alpha_1},\ldots,S_{\alpha_r})$ is equivalent to the category of representations of the quiver $Q(S)$ having $r$ vertices $s_1,\ldots,s_r$ and $\ext(\alpha_i,\alpha_j)=\dim\Ext(S_{\alpha_i},S_{\alpha_j})$ arrows from $s_i$ to $s_j$.
\begin{thm}[{\cite[section 2]{sc2},\cite[Theorem 2.38]{dw2}}]\label{root}
Let $S=(S_{\alpha_1},\ldots,S_{\alpha_r})$ be a reduced exceptional sequence. Then $\alpha=\sum_{i=1}^rk_i\alpha_i$ is a root of $Q$ if and only if $(k_1,\ldots,k_r)$ is a root of $Q(S)$.
\end{thm}
We will need the following lemma:
\begin{lem}\label{insu}
Let $N$ be an indecomposable representation and $\alpha$ be a real Schur root such that $\udim  N$ is no real root and $N\in S_{\alpha}^{\perp}$ (resp. $N\in {^\perp} S_{\alpha}$) . If $\Hom(N,S_{\alpha})\neq 0$ (resp. $\Hom(S_{\alpha},N)\neq 0$) we have $N\twoheadrightarrow S_{\alpha}$ (resp. $S_{\alpha}\hookrightarrow N$).
\end{lem}
\begin{proof}

We only consider the first case. The second case can be proved analogously.
Since $\Ext(S_{\alpha},N)=0$, by Lemma \ref{happelringel} every morphism $f:N\rightarrow S_{\alpha}$ is either injective or surjective. If $f$ were injective, we would get a surjection $\Ext(S_{\alpha},N)\twoheadrightarrow\Ext(N,N)$ which contradicts $\Ext(S_{\alpha},N)=0$ because $N$ has self-extensions. 
\end{proof}

\subsection{Ringel's reflection functor}\label{ringelrefl}
As our construction is motivated by Ringel's reflection functor, we review several results of \cite[section 1]{rin}. For a fixed exceptional representation $S$ and a full subcategory $\mathcal{C}$ of $\Rep(Q)$, let $\mathcal{C}/S$ be the category which has the same objects as $\mathcal{C}$ and the same maps modulo those factoring through $S^{n}$ for some $n\in\Nn$. We will also consider the following full subcategories of $\Rep(Q)$:
\begin{enumerate}
\item $\mathcal{M}^{-S}=\{X\in\Rep(Q)\mid\Hom(X,S)=0\}$; 
\item $\mathcal{M}_{-S}=\{X\in\Rep(Q)\mid \Hom(S,X)=0\}$;
\item $\mathcal{M}^S$, the category of representations $X\in\Rep(Q)$ with $\Ext(S,X)=0$ such that, moreover, there does not exist a direct summand of $X$ which can be embedded into a direct sum of copies of $S$;
\item $\mathcal{M}_S$, the category of representations $X\in\Rep(Q)$ with $\Ext(X,S)=0$ such that, moreover, no direct summand of $X$ is a quotient of a direct sum of copies of $S$.

\end{enumerate}
Following \cite[Lemma 2]{rin}, for a fixed representation $X\in\mathcal{M}^S$, a basis $\mathcal{B}:=\{\varphi_1,\ldots,\varphi_n\}$ of $\Hom(X,S)$ induces an exact sequence 
\[\ses{X^{-S}}{X}{S^{ n}}\]
such that the induced sequences $e_1,\ldots,e_n$ form a basis of $\Ext(S,X^{-S})$. Moreover, we have $X^{-S}\in\mathcal{M}^{-S}$. The other way around, if $Y\in\mathcal{M}^{-S}$ and $\{e_1,\ldots,e_n\}$ is a basis of $\Ext(S,Y)$, there exists an induced sequence
\[\ses{Y}{Y^S}{S^{n}}\]
such that $Y^{S}\in\mathcal{M}^S$. We can proceed similarly for $X\in\mathcal M_S$ and $Y\in\mathcal M_{-S}$. The following theorem summarizes the results of \cite[Section 1]{rin}:
\begin{thm}\label{ringel}
\begin{enumerate} 
\item  There exists an equivalence of categories given by the functor $F:\mathcal{M}^S/S\rightarrow\mathcal{M}^{-S},\,X\mapsto X^{-S}.$
\item There exists an equivalence of categories given by the functor $G:\mathcal{M}_S/S\rightarrow\mathcal{M}_{-S},$ $X\mapsto X_{-S}$.
\item There exist equivalences $\Psi:\mathcal{M}^{S}_{-S}\rightarrow\mathcal{M}^{-S}_{S}$ and $\Phi:\mathcal{M}^S_S/S\rightarrow\mathcal{M}^{-S}_{-S}$ induced by composing the functors from above.
\end{enumerate}
\end{thm}

\subsection{Recursive construction of indecomposable Kronecker representations}\label{Krone}
We review a special case of the functor investigated in \cite{wei3} which can be used to construct indecomposable representations recursively. To do so, let $M=(M_1,M_2)$ be a pair of Schurian representations of $Q$ and let $n_{12}:=\dim\Ext(M_1,M_2)$ and $n_{21}:=\dim\Ext(M_2,M_1)$. For $i\neq j$, we fix subsets 
$$\mathcal{B}_{ij}=\{\chi_1^{ij},\ldots,\chi_{n_{ij}}^{ij}\}\subset\bigoplus_{\rho\in Q_1}\Hom_k((M_i)_{s(\rho)},(M_j)_{t(\rho)})$$ such that the corresponding residue classes are a basis of $\Ext(M_i,M_j)$. Then we consider the quiver $Q(M)$ having vertices $\{m_1,m_2\}$ and $n_{ij}$ arrows from $m_i$ to $m_j$. For a representation $X$ of $Q(M)$, we define $\tilde X_{i,q}:=(M_i)_q\otimes_k X_{m_i}$ where $q\in Q_0$ and $i\in\{1,2\}$. 

This yields a functor $F_M:\Rep(Q(M))\rightarrow \Rep(Q)$ which is defined on objects as follows: we define a representation $F_MX$ of $Q$ by the vector spaces
\[(F_MX)_q=\tilde X_{1,q}\oplus\tilde X_{2,q}\text{ for all }q\in Q_0\]
and for $\rho:q\rightarrow q'\in Q_1$ we define linear maps $(F_MX)_{\rho}=\tilde X_{1,q}\oplus\tilde X_{2,q}\to\tilde X_{1,q'}\oplus\tilde X_{2,q'}$ by
\[((F_MX)_{\rho})_{i,i}=(M_i)_{\rho}\otimes_k\mathrm{id}_{X_{m_i}}:\tilde X_{i,q}\rightarrow \tilde X_{i,q'}\]
and 
\[((F_MX)_{\rho})_{i,j}=\sum_{l=1}^{n_{ji}}(\chi^{ji}_l)_{\rho}\otimes_kX_{\chi^{ji}_l}:\tilde X_{j,q}\rightarrow \tilde X_{i,q'}\]
for $i\neq j$. 
Then we have the following theorem:
 
\begin{thm}[{\cite[Theorem 3.3]{wei3}}]\label{thm1}
If $M=(M_1,M_2)$ is a pair of Schurian representations such that $\Hom(M_1,M_2)=\Hom(M_2,M_1)=0$, the functor $F_M$ is a fully faithful embedding. In particular, $F_MX$ is indecomposable if and only if $X$ is indecomposable. 

\end{thm}

We want to apply this to a fixed generalized Kronecker quiver which we denote by $K(m)=(\{q_0,q_1\},\{\rho_i:q_0\to q_1\mid i=1,\ldots,m\})$ with $m\geq 3$.
\begin{defi}\label{homorth} A pair $((d_s,e_s),(d,e))$ of Schur roots of $K(m)$ is called $\Hom$-orthogonal if $\hom((d_s,e_s),(d,e))=0$ and $\hom((d,e),(d_s,e_s))=0$.
\end{defi}

For a fixed $\Hom$-orthogonal pair $((d_s,e_s),(d,e))$, we can in particular construct Schurian representations of dimension $(d_s,e_s)+(d,e)$ using Theorem \ref{thm1}. The following lemma is a consequence of the considerations of \cite[section 4.3]{wei2} and it is needed to prove the existence of such pairs.

\begin{lem} \label{l:decomp}
Every non-simple coprime Schur root $(\hat d,\hat e)$ of $K(m)$ can be decomposed into $(\hat d,\hat e)=(d_s,e_s)+k(d,e)$ where $k\geq 1$ and $(d_s,e_s)$ and $(d,e)$ are  coprime Schur roots such that 
\begin{enumerate}
	\item $e_sd-ed_s=1$;
	\item $1\leq d_s\leq d$; 
	\item $1\leq e_s\leq e$ if $d_s\neq 1$ and $e_s=e+1$ if $d_s=1$.
\end{enumerate}
\end{lem}

Actually, it is shown in \cite[section 4.3]{wei2} that such a decomposition can be constructed recursively. For $2\leq n\leq m$, we have that the pair $((d_s,e_s),(d,e))=((1,n),(1,n-1))$ satisfies the numerical conditions (1), (2), and (3) of Lemma \ref{l:decomp}. It is also shown that, if $((d_s',e_s'),(d',e'))$ satisfies these numerical conditions, then the pair consisting of \begin{equation}\label{eq2}(d_s,e_s)=(d_s',e_s')+k(d',e'),\quad (d,e)=(d_s',e_s')+(k+1)(d',e')\end{equation}
does so, too. Finally, it is proved that every coprime root $(\hat d,\hat e)$  with $\hat d,\hat e\geq 2$ can obtained in this way. The following proposition ensures that this decomposition gives rise to a $\Hom$-orthogonal pair.

\begin{pro}\label{hom0}
Assume that $(d_s,e_s)$ and $(d,e)$ are coprime Schur roots of $K(m)$ which fulfill the conditions (1), (2), and (3) of Lemma \ref{l:decomp}. Then the pair $((d,e),(d_s,e_s)+k(d,e))$ is $\Hom$-orthogonal for every $k\geq 0$.
\end{pro}

\begin{proof}
Since the notion of stability is not needed elsewhere in this paper, we do not give details and refer to \cite{rei3}.
We consider the slope function $\mu:\Nn Q_0\backslash\{0\}\to\Qn$ defined by $\mu(d,e)=\frac{d}{d+e}$. The stability condition induced by this slope function is equivalent to the one induced by the Euler form, see also \cite[section 6]{sch}. Thus, since all roots of $K(m)$ with $m\geq 3$ are Schurian, a general representation of this dimension is stable. As the pair of roots under consideration satisfies the glueing condition introduced in \cite[section 4.3]{wei2}, we obtain $\mu((d_s,e_s)+k(d,e))<\mu(d,e)$. We thus already have $\hom((d,e),(d_s,e_s)+k(d,e))=0$, see also \cite[Lemma 4.2]{rei3}. 

In order to show that $\hom((d_s,e_s)+k(d,e),(d,e))=0$, it suffices to construct two representations $M$ and $N$ of dimensions $(d_s,e_s)+k(d,e)$ and $(d,e)$ respectively such that $\Hom(M,N)=0$. Since $\hom((d,e),(d,e))=0$ if $(d,e)$ is imaginary by \cite[Theorem 3.5]{sch}, it suffices to prove that $\hom((d_s,e_s),(d,e))=0$. Indeed, for $M$ we can take a direct sum of representations of dimensions $(d_s,e_s)$ and $(d,e)$ respectively. If $(d_s,e_s)=(1,n)$, $(d,e)=(1,n-1)$, we have $n\leq m-1$ because $(d_s,e_s)$ is imaginary. Then we can easily construct representations $M$ and $N$ such that $\Hom(M,N)=0$. For instance, the indecomposable tree modules
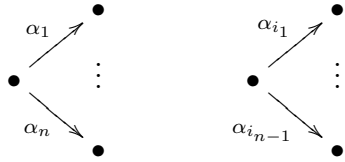
\[
\begin{xy}\xymatrix@R10pt@C20pt{&\bullet&&&\bullet\\\bullet\ar[ru]^{\alpha_{1}}\ar[rd]_{\alpha_{n}}&\vdots&&\bullet\ar[ru]^{\alpha_{i_1}}\ar[rd]_{\alpha_{i_{n-1}}}&\vdots\\&\bullet&&&\bullet}
\end{xy}\]
such that $i_1\notin \{1,\ldots,n\}$, which is possible because $n\neq m$,
satisfy this property. If we keep in mind equation (\ref{eq2}) and that $(d_s,e_s)$ is also an imaginary root, again taking direct sums, the claim follows by induction. 
\end{proof}

\begin{kor}
We can decompose every non-simple root $(\hat d,\hat e)$ of a generalized Kronecker quiver $K(m)$ as $(\hat d,\hat e)=(d_s,e_s)+(k+1)(d,e)$ where $((d,e),(d_s,e_s)+k(d,e))$ is $\Hom$-orthogonal and $k\geq 0$.
\end{kor}
\begin{proof}
It remains to consider the case when $(\hat d,\hat e)$ is not coprime, say $(\hat d,\hat e)=s(\tilde d,\tilde e)$ with $s\geq 2$ and $(\tilde d,\tilde e)$ coprime. But then we can decompose $(\tilde d,\tilde e)=(d'_s,e_s')+(k+1)(d',e')$ such that $((d',e'),(d_s',e_s')+k(d',e'))$ is $\Hom$-orthogonal. Finally, we set $(d_s,e_s)=s(d_s',e_s')$ and $(d,e)=s(d',e')$.
\end{proof}

%



\subsection{Recollections on symmetric functions}

We state some definitions and facts on symmetric functions. As a reference, we suggest Macdonald's \cite{mac} or Manivel's book \cite{man}.

Let $\Lambda$ be the ring of symmetric functions in variables $x_1,x_2,\ldots$ over $\Zn$. There are several bases (as an abelian group) of this ring that will be important for us. Let $\lambda = (\lambda_1,\lambda_2,\ldots)$ be a partition. The monomial symmetric function $m_\lambda$ is defined as the sum 
$$
	m_\lambda = \sum_\alpha x^{\alpha}
$$ 
of monomials $x^\alpha = x_1^{\alpha_1}x_2^{\alpha_2}\ldots$ over all distinct permutations $\alpha = (\alpha_1,\alpha_2,\ldots)$ of $\lambda$. The monomial symmetric functions form a basis of $\Lambda$ over the integers. For a non-negative integer $r$, the $r$\textsuperscript{th} elementary symmetric function $e_r$ and the $r$\textsuperscript{th} complete symmetric function $h_r$ are defined as
$$
	e_r = m_{1^r} = \sum_{i_1 < \ldots < i_r} x_{i_1}\ldots x_{i_r}\ \text{ and }\ h_r = \sum_{ \lambda \vdash r} m_\lambda.
$$
In the above context, $\lambda \vdash r$ means that $\lambda$ is a partition of $r$. The partition $1^r$ is $(1,\ldots,1)$, the number $1$ repeated $r$ times. For a partition $\lambda$, we set $e_\lambda = e_{\lambda_1}e_{\lambda_2}\ldots$ and $h_\lambda = h_{\lambda_1}h_{\lambda_2}\ldots$ and obtain two more bases for $\Lambda$ over the integers. The generating functions $E(t) = \sum_{r \geq 0} e_rt^r$ and $H(t) = \sum_{r \geq 0} h_rt^r$ are related by the identity $E(-t)H(t) = 1$ (cf.\ \cite[Chapter I, (2.6)]{mac}). From this identity, we deduce that for a partition $\lambda$, we have
$$
	\det(h_{\lambda_i-i+j}) = \det(e_{\lambda'_j-j+i}),
$$
where the matrices are sufficiently large and $e$ and $h$ with a negative subscript are interpreted as zero \cite[Chapter I, (3.4), (3.5)]{mac}. We write $\lambda'$ for the conjugate partition of $\lambda$. We define the so-called Schur function $s_\lambda$ as $\det(h_{\lambda_i-i+j})$. They yield yet another basis of $\Lambda$. We will use the fact that the transition matrix expressing the elementary symmetric functions in terms of the Schur functions is given by the Kostka numbers. More precisely:

\begin{lem}[{\cite[Chapter I, (6.4), (6.5)]{mac}}] \label{kostka}
	For a partition $\mu$ of $r$, we have
	$$
		e_{\mu} = \sum_\lambda K_{\lambda,\mu} s_{\lambda'},
	$$
	the sum ranging over partitions $\lambda$ of $r$ and $K_{\lambda,\mu}$ is the number of semi-standard Young tableaux of shape $\lambda$ and weight $\mu$ -- the so-called Kostka number. The matrix $(K_{\lambda,\mu})$ is strictly upper unitriangular with respect to the dominance order.
\end{lem}

The product $s_\lambda \cdot s_\mu$ in $\Lambda$ is a non-negative linear combination $\sum_\nu N_{\lambda,\mu,\nu} s_\nu$, where the numbers $N_{\lambda,\mu, \nu}$ are given by the Littlewood--Richardson rule \cite[Chapter I, \S 9]{mac}. Of particular interest for us is the case where $\mu = 1^r$.

\begin{lem}[Pieri's rule, {\cite[1.2.5]{man}}] \label{pieri}
	For a partition $\lambda$, we have $s_\lambda \cdot s_{1^r} = s_\lambda \cdot e_r = \sum_\nu s_\nu$, where $\nu$ ranges over all partitions arising from $\lambda$ by adding $r$ new boxes, at most one per row.
\end{lem}

\subsection{A reminder on intersection theory} \label{intersection_theory}

We briefly recall the basic notions of intersection theory and give some results which are necessary for the present work. Our exposition is far from being complete. As a main reference on intersection theory, we recommend Fulton's book \cite{ful}.

The Chow group $A_*(X) = \bigoplus_{n \geq 0} A_n(X)$ of an algebraic\footnote{Following \cite[B.1.1]{ful}, a $k$-scheme is called algebraic if it is separated and of finite type over $\Spec k$.} $k$-scheme $X$ is the group of cycles\footnote{An $n$-cycle is a $\mathbb{Z}$-linear combination of $n$-dimensional subvarieties of $X$.} up to rational equivalence\footnote{Rational equivalence is an appropriate generalization of the notion of linear equivalence of Weil divisors; see \cite[1.3]{ful}}. This group possesses functorial properties: let $f: Y \to X$ be a morphism of algebraic schemes. If $f$ is proper then there exists a push-forward $f_*: A_*(Y) \to A_*(X)$ which is a homomorphism of graded abelian groups \cite[1.4]{ful}. If $f$ is flat of relative dimension $r$ we can define a pull-back $f^*: A_*(X) \to A_{*+r}(Y)$, a homomorphism of abelian groups which increases degrees by $r$ (see \cite[1.7]{ful}). In case $f$ is a regular embedding of codimension $d$ (or, more generally, an lci morphism) there is a Gysin pull-back $f^*: A_*(X) \to A_{*-d}(Y)$ (Chapter 6 of \cite{ful}).

Assuming that $X$ is a non-singular variety of dimension $N$, it is possible to construct a multiplication on the Chow group of $X$ (as described in \cite[Chapter 8]{ful}). When defining $A^i(X) = A_{N-i}(X)$, we obtain a graded ring $A^*(X)$, the Chow ring of $X$. The pull-back induced by a morphism of non-singular varieties (which is automatically lci) is a homomorphism of graded rings.

Chow rings possess a theory of Chern classes (as described axiomatically in Grothendieck's article \cite{gro}); the $i$\textsuperscript{th} Chern class of a vector bundle $E$ on $X$ is a class $c_i(E) \in A^i(X)$. Fulton defines Chern classes in \cite[3.2]{ful}. If the Chern polynomial $c_t(E) = 1 + c_1(E)t + c_2(E)t^2 + \ldots$ (which is in fact a polynomial as $c_i(E) = 0$ if $i$ exceeds the rank of $E$) is factored as $c_t(E) = \prod_i (1+\xi_it)$ then the $\xi_i$'s are called the Chern roots of $E$. Note that the Chern classes are the elementary symmetric polynomials in the Chern roots.

There is also a localized version of Chern classes; for us, the localized top Chern class is important (see \cite[14.1]{ful}). For a section $s: X \to E$ of a vector bundle of rank $r$ on a purely $N$-dimensional algebraic scheme, there exists a class $\mathbb{Z}(s) \in A_{N-r}(Z(s))$ in the Chow group of the zero locus of $s$. Its push-forward along the closed embedding $Z(s) \to X$ equals $c_r(E)$ (whence the name localized top Chern class) and if $s$ is a regular section then $\mathbb{Z}(s)$ agrees with the cycle $[Z(s)]$ associated with the scheme\footnote{The cycle $[Z] \in A_*(Z)$ associated with a scheme $Z$ is defined in \cite[1.5]{ful}.} $Z(s)$ (cf.\ \cite[Proposition 14.1]{ful}).

We will make extensive use of an explicit description of the Chow ring of the Grassmannian $\Gr_d(k^n)$ in subsection \ref{intsec}. The results described here can be found in \cite[14.6]{ful}. Let $\UU$ be the universal rank $d$ subbundle of the trivial bundle of rank $n$ on $\Gr_d(k^n)$ and let $\QQ$ be the cokernel of this inclusion. We define the class 
$$
	\Delta_\lambda = \det( c_{\lambda_i+j-i}(\QQ) ) = \det(c_{\lambda'_j+i-j}(\UU^\vee));
$$ 
it agrees with the Schur function $s_\lambda$ evaluated at the Chern roots of $\UU^\vee$ (note that the Chern classes of $\QQ$ are the complete symmetric functions in the Chern roots of $\UU^\vee$). It is easy to see that $\Delta_\lambda = 0$, unless $\lambda$ is contained in $(n-d)^d$ (i.e.\ the length of $\lambda$ is at most $d$ and $\lambda_1 \leq n-d$). Given a flag $U_1 \sub \ldots \sub U_d$ of $k^n$ with $\dim U_i = n-d+i-\lambda_i$, the cycle associated with the closed subscheme
$$
	\Omega(U_*) = \{ U \in \Gr_d(k^n) \mid \dim(U \cap U_i) \geq i \text{ (all $i = 1,\ldots,d$)} \}
$$
is $\Delta_\lambda$ (``Giambelli's formula'' \cite[Proposition 14.6.4]{ful}). Most important for us is that these classes provide a basis of the Chow ring.

\begin{pro}[Basis theorem, {\cite[Proposition 14.6.5]{ful}}] \label{basis_thm}
	As an abelian group, $A^*(\Gr_d(k^n))$ is free with a basis given by the classes $\Delta_\lambda$, where $\lambda$ ranges over all partitions $\lambda = (\lambda_1,\ldots,\lambda_d)$ with $n-d \geq \lambda_1 \geq \ldots \geq \lambda_d \geq 0$.
\end{pro}

As $\Delta_\lambda$ equals the Schur function $s_\lambda$ evaluated at the Chern roots of $\UU^\vee$, the product of the basis elements $\Delta_\lambda$ and $\Delta_\mu$ is given by the Littlewood--Richardson rule $\Delta_\lambda \cdot \Delta_\mu = \sum_\nu N_{\lambda,\mu,\nu} \Delta_\nu$ and, in particular, by Pieri's rule, we have
$$
	\Delta_\lambda \cdot c_r(\UU^\vee) = \sum_\nu \Delta_\nu,
$$
the summation ranging over the same $\nu$'s as in Lemma \ref{pieri}. Here, we can, of course, restrict ourselves to those $\nu$'s contained in $(n-d)^d$.

We would finally like to mention that the push-forward $\pi_* \Delta_\lambda$ of such a basis element along the structure map $\pi: \Gr_d(k^n) \to \Spec k$ is $1$ if $\lambda = (n-d)^d$ and vanishes for every other partition $\lambda$. This fact is used in the proof of the ``duality theorem'' \cite[Proposition 14.6.4]{ful}. We see from the definition that $\Delta_{(n-d)^d} = c_{n-d}(\QQ)^d$.



\section{Quivers with three vertices}\label{three}
\noindent One of the main goals in representation theory is to classify the indecomposable representations (of a fixed dimension). Since this is a very difficult problem in general, a step towards it is to construct families of indecomposable representation of a fixed dimension. For a fixed Schur root $\alpha$, one possibility is to use the methods of \cite{wei} and \cite{wei4} respectively. The main focus of this section is on  representations of quivers with three vertices which have certain non-Schurian roots as dimension vectors. We establish a connection between points of subvarieties of Grassmannians attached to a fixed non-Schurian root and representations which have this root as dimension vector. In many cases these representations turn out to be indecomposable. We restrict to quivers with three vertices (and without oriented cycles); however the results can be generalized to certain roots of arbitrary quivers, see Remark \ref{generalization}.  

\subsection{Non-Schurian indecomposables of quivers with three vertices}\label{haupt}
We fix a vector $\un{m}:=(m_{12},m_{13},m_{23})\in\mathbb N^3$.
Let $Q(\un m)$ be the quiver
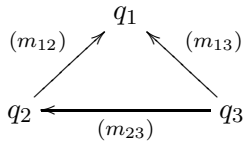
\[
\begin{xy}\xymatrix{&q_1&\\q_2\ar[ru]^{(m_{12})}&&q_3\ar[lu]_{(m_{13})}\ar[ll]^{(m_{23})}}
\end{xy}\]
where $m_{ij}$ in brackets indicates the number of arrows between the corresponding vertices. We denote the arrows by $\rho_1^i:q_2\to q_1$ for $i=1,\ldots,m_{12}$, $\rho_2^i:q_3\to q_1$ for $i=1,\ldots,m_{13}$ and $\rho_3^i:q_3\to q_2$ for $i=1,\ldots,m_{23}$. 
Recall from \cite[section 6]{dw} that the canonical decomposition of a dimension vector $\alpha$ of $Q(\un{m})$ either consists of (at most three) multiples of real Schur roots, one imaginary Schur root or both a multiple of an imaginary Schur root and a multiple of a real Schur root. 

Note that, if $\alpha$ is a Schur root, it is shown in \cite[Theorem 3.1.7]{wei4} that the methods of \cite{wei} can be used to construct a $(1-\Sc{\alpha}{\alpha})$-parameter family of isomorphism classes of indecomposable representations of dimension $\alpha$. Actually, this construction is independent of the quiver $Q$.

We concentrate on the last case. Then the imaginary Schur root appears  with multiplicity one if its Euler form is negative. As the multiple of an imaginary Schur root is again imaginary (but not Schurian if the Euler form is zero), we say, by abuse of notation, that its canonical decomposition consists of an imaginary root and a real Schur root.

\begin{lem} \label{l:can_decomp}
Let $\alpha$ be a non-Schurian root of $Q(\un m)$. Then the canonical decomposition of $\alpha$ is $\alpha=\alpha_1^{d_1}\oplus\hat\alpha$ or $\alpha=\hat\alpha\oplus\alpha_1^{d_1}$ where $\hat\alpha$ is an imaginary root and $\alpha_1$ is a real Schur root. Moreover, we have $\smash{\hat\alpha}=\alpha_2^{d_2}+\alpha_3^{d_3}$ where $\alpha_2,\alpha_3$ are the two simple roots in $\alpha_1^{\perp}$ or in ${^\perp}\alpha_1$ and $(d_3,d_2)$ is an imaginary Schur root of $K(\ext(\alpha_3,\alpha_2))$ or $K(\ext(\alpha_2,\alpha_3))$.
\end{lem}
\begin{proof}
First note that the maximal length of an exceptional sequence is equal to the number of vertices of the quiver. From this we get that a dimension vector $\alpha$ whose canonical decomposition contains more than one exceptional root cannot contain an imaginary root. Thus $\alpha$ cannot be a root because $\Sc{\alpha}{\alpha}\geq 2$ in this case. This shows the first part. Moreover, together with Theorem \ref{perpcat} this shows that there are only two simple roots $\alpha_2,\alpha_3$ in $\alpha_1^\perp$ and $^ \perp\alpha_1$. Since $Q(\alpha_1^\perp)$ and $Q({^\perp}\alpha_1)$ are forced to be  generalized Kronecker quivers, we either have $\ext(\alpha_2,\alpha_3)=0$ or $\ext(\alpha_3,\alpha_2)=0$. Thus the imaginary root $\hat\alpha=\alpha_2^{d_2}+\alpha_3^{d_3}$ corresponds to the imaginary root $(d_3,d_2)$ of $K(\ext(\alpha_3,\alpha_2))$ or $K(\ext(\alpha_2,\alpha_3))$.
\end{proof}

Since the two cases are dual, we will mostly restrict to the first one.  We also assume that $\ext(\alpha_2,\alpha_3)=0$ which implies that $Q(\alpha_1^\perp)=K(\ext(\alpha_3,\alpha_2))$.  This also means that $\alpha_3$ corresponds to the source $q_0$ and $\alpha_2$ to the sink $q_1$ of $K(\ext(\alpha_3,\alpha_2))$.  This obviously implies $\ext(\alpha_3,\alpha_2)\neq 0$.

\begin{defi}
Let $\alpha$ be a root of $Q(\un{m})$ with canonical decomposition $\alpha=\alpha_1^{d_1}\oplus\hat\alpha$. The unique decomposition of $\alpha$ into exceptional roots $\alpha=\alpha_1^{d_1}+\alpha_2^{d_2}+\alpha_3^{d_3}$ where $\alpha_2$ and $\alpha_3$ are the simple roots of $\alpha_1^{\perp}$ with $\ext(\alpha_2,\alpha_3)=0$ is called the canonical exceptional decomposition of $\alpha$.
\end{defi}
Note that $(\alpha_1,\alpha_2,\alpha_3)$ is a complete exceptional sequence in $Q(\un{m})$. The case $d_1=0$ is only important in section \ref{glueing}; then $\alpha\in\alpha_1^\perp$. 
Recall that $S_\beta$ denotes the unique indecomposable representation of dimension $\beta$ (up to isomorphism) if $\beta$ is a real root, whereas, as before, $M_\beta$ can be any representation of dimension $\beta$. 

We fix a non-Schurian root $\alpha$ with canonical exceptional decomposition as above. The main aim of this paper is to construct a $(1-\Sc{\alpha}{\alpha})$-parameter family of isomorphism classes of (indecomposable) representations $M_{\hat\alpha}$ such that $\dim\Hom(M_{\hat\alpha}, S_{\alpha_1})\geq\Sc{\hat\alpha}{\alpha_1}+d_1$. Then we have 
\begin{equation}\label{eq1}
\dim\Ext(M_{\hat\alpha},S_{\alpha_1})=-\Sc{\hat\alpha}{\alpha_1}+\dim\Hom(M_{\hat\alpha},S_{\alpha_1})\geq d_1.
\end{equation} 
We say that two exact sequences $e\in\Ext(M,N)$ and $e'\in\Ext(M',N')$ are isomorphic if they give rise to a commutative diagram where the rows are $e$ and $e'$ and where the vertical morphisms are isomorphisms. The main construction of isomorphism classes of indecomposables of dimension $\alpha$ relies on the following proposition:

\begin{pro} \label{p:ses_indec}
Let $M_{\hat\alpha} \in \alpha_1^\perp$ and let $U\subset \Ext(M_{\hat\alpha},S_{\alpha_1})$ be a $d_1$-dimensional subspace. Then the middle term of the induced short exact sequence \[e_U:\ses{S_{\alpha_1}^{d_1}}{M_{\alpha}}{M_{\hat\alpha}}\]
is indecomposable whenever $M_{\hat\alpha}$ is indecomposable. Moreover, given another $M'_{\hat\alpha} \in \alpha_1^\perp$ and another $d_1$-dimensional subspace $U'\subset \Ext(M'_{\hat\alpha},S_{\alpha_1})$, the middle terms $M_{\alpha}$ and $M'_{\alpha}$ of $e_U$ and $e_{U'}$ are isomorphic if and only if $e_U$ and $e_{U'}$ are isomorphic.
\end{pro}
\begin{proof}
The first part is proved in \cite[Proposition 3.13]{wei} which is based on Theorem \ref{perpcat}. As $\Hom(S_{\alpha_1},M_{\hat\alpha})=\Hom(S_{\alpha_1},M_{\hat\alpha}')=0$, using the Snake Lemma together with the universal property of the (co)kernel, an isomorphism $M_\alpha\cong M'_{\alpha}$ induces isomorphisms $M_{\hat\alpha}\cong M_{\hat\alpha}'$ and $S_{\alpha_1}^{d_1}\cong S_{\alpha_1}^{d_1}$ which means that $e_U$ and $e_{U'}$ are isomorphic. The other direction is obvious.
\end{proof}

In order to construct representations $M_{\hat\alpha}\in\alpha_1^{\perp}$ which satisfy inequality (\ref{eq1}), we first consider the exceptional representation $S_{\delta}$ obtained as the middle-term of the exact sequence induced by a basis of $\Ext(S_{\alpha_3},S_{\alpha_2})$, i.e.:
\[\ses{S_{\alpha_2}^{\ext(\alpha_3,\alpha_2)}}{S_{\delta}}{S_{\alpha_3}}.\]
Since $(d_3,d_2)$ is an imaginary Schur root of $K(\ext(\alpha_3,\alpha_2))$, we have $r:=\ext(\alpha_3,\alpha_2)d_3-d_2\geq 1$. The next step is to study under which conditions there exist commutative diagrams with exact rows of the form
\begin{equation}
\begin{xy}\label{diag}\xymatrix@R30pt@C29pt{0\ar[r]&S_{\alpha_2}^r\ar[r]^{i_1}\ar[d]^{f_1}&S_{\delta}^{d_3}\ar[d]^{f_2}\ar[r]^{\pi_1}&M_{\hat\alpha}\ar[r]\ar[d]^{f_3}&0\\0\ar[r]&S_{\alpha_1}^t\ar[r]^{i_2}&S_{\alpha_1}^s\ar[r]^{\pi_2}&S_{\alpha_1}^{s-t}\ar[r]&0}
\end{xy}
\end{equation}
such that the morphisms $f_i$ are of maximal rank (as vector space homomorphisms). Even if it is of maximal rank, $f_3$ cannot be injective because $\hat\alpha$ is an imaginary Schur root and $M_{\hat\alpha}\in\alpha_1^{\perp}$ so that the same proof as the one of Lemma \ref{insu} shows that $f_3$ is forced to be surjective in this case. If $f_3$ is surjective, we have $\dim\Hom(M_{\hat\alpha},S_{\alpha_1})\geq s-t$. This is induced by the injection $\Hom(S_{\alpha_1}^{s-t},S_{\alpha_1})\to \Hom(M_{\hat\alpha},S_{\alpha_1})$. 
In particular, if $s-t\geq \Sc{\hat\alpha}{\alpha_1}+d_1$, we would have constructed a representation $M_{\hat\alpha}$ with $\dim\Ext(M_{\hat\alpha},S_{\alpha_1})\geq d_1$. 
\begin{rem}\label{rem1}
If $d_3=1$, the representations $M_{\hat\alpha}$, which are obtained as the cokernel of the upper row as above, are automatically indecomposable. This is because every representation $X$ of $K(\ext(\alpha_3,\alpha_2))$ of dimension $(1,r)$ is Schurian as soon as $\dim\sum_{i=1}^{\ext(\alpha_3,\alpha_2)}X_{\rho_i}(X_{q_0})=r$. This is ensured by the injectivity of $i_1$.
\end{rem}

The following questions will be topic of the next sections:
\begin{que}\label{mainque}
\begin{enumerate}
\item How can we describe representations $M_{\hat\alpha}$ obtained by commutative diagrams of the form (\ref{diag})?
\item How can we assure that the morphisms $f_i$ are of maximal rank?

\item What can we say if $d_3\geq 2$ concerning indecomposability?

\end{enumerate}
\end{que}
%
%

\subsection{Subvarieties of Grassmannians induced by non-Schurian roots}\label{subvarieties}
\noindent
We fix a non-Schurian root $\alpha$ of $Q(\un m)$ and keep the notation of section \ref{haupt}. The aim is to associate a Grassmannian and two subvarieties of this Grassmannian with this root. Moreover, we want to study their intersection and show that every point in this intersection gives rise to a representation $M_{\hat\alpha}$ which satisfies inequality (\ref{eq1}). This gives an explicit answer to the first two parts of Question \ref{mainque}. Before we can define and study these varieties, we need some general observations concerning the canonical exceptional decomposition of $\alpha$.

\begin{lem}\label{kipp1}
Let $\alpha$ be a non-Schurian root of $Q(\un m)$ with canonical decomposition $\alpha=\alpha_1^{d_1}\oplus\hat\alpha$ and with
canonical exceptional decomposition $\alpha=\alpha_1^{d_1}+\alpha_2^{d_2}+\alpha_3^{d_3}$. Then we have $\hom(\alpha_3,\alpha_2)=0$, $\hom(\alpha_3,\alpha_1)=0$ and $\ext(\alpha_2,\alpha_1)=0$.
\end{lem}
\begin{proof}
The first equation follows because $\alpha_2$ and $\alpha_3$ are simple in $\alpha_1^{\perp}$. We even have that a general representation of dimension $\hat\alpha$ has a filtration \[\ses{S_{\alpha_2}^{d_2}}{M_{\hat\alpha}}{S_{\alpha_3}^{d_3}}.\]
Thus since $\ext(\hat\alpha,\alpha_1)=0$, we immediately get $\ext(\alpha_2,\alpha_1)=0$.

Since $\alpha$ is not a Schur root, it is neither a preprojective nor a preinjective root. Thus by Lemma \ref{injprojcanon}, we have that $\alpha_1$ is neither preprojective nor preinjective and thus regular. Since $Q(\un m)$ has three vertices, by  \cite[Theorem 2.6]{hos}, it follows that $\alpha_1$ is a quasi-simple root.  In particular, there exists an almost split sequence with indecomposable middle-term
\[\ses{\tau S_{\alpha_1}}{M_\gamma}{S_{\alpha_1}}\]
such that $M_\gamma\in{\alpha_1}^{\perp}\cong{^\perp}(\tau\alpha_1)$. Note that the last equivalence is proved in \cite[Theorem 2.1]{sc2}. 
Again by \cite[Theorem 2.6]{hos}, it follows that $\gamma$ is an imaginary root and that we have $\Ext(M_\gamma,M_\gamma)\cong \Ext(M_{\gamma},S_{\alpha_1})\neq 0$. Thus, since $\ext(\alpha_2,\alpha_1)=0$ and $\gamma=\alpha_2^s+\alpha_3^t$ for some $s,t\geq 0$, we necessarily have $\ext(\alpha_3,\alpha_1)\neq 0$ and thus $\hom(\alpha_3,\alpha_1)=0$ by Theorem \ref{schofield}.
%
\end{proof}
	For a fixed representation $M$ and for a fixed dimension vector $\beta$, we define the quiver Grassmannian 
	$$\mathrm{Gr}_{\beta}(M)=\{U\in R_{\beta}(Q)\mid U\subset M\},$$
	which is a closed subvariety of $\prod_{i\in Q_0}\mathrm{Gr}_{\beta_i}(M_i)$. The only required statement concerning quiver Grassmannians is the following: 
\begin{lem}
For an exceptional root $\beta$, we have
$$\Gr_{t\beta}(S_\beta^s)\cong\Gr_t(k^s)$$
for $0\leq t\leq s$.
 
\end{lem}

\begin{proof}
	Since $\beta$ is an exceptional root, we have $\Sc{\gamma}{\beta}-\Sc{\beta}{\gamma}>0$ for every $0\neq\gamma\neq\beta$ with $\Gr_\gamma(S_\beta)\neq\emptyset$, see \cite[Theorem 6.1]{sch}. Inductively, we can deduce from the surjective morphism
	\begin{equation} \label{e:Gr}
		\Gr_{\gamma}(S_\beta \oplus S_\beta^{s-1}) \to \bigsqcup_{\delta + \epsilon = \gamma} \Gr_{\delta}(S_\beta) \times \Gr_{\epsilon}(S_\beta^{s-1})
	\end{equation}
	(see \cite[section 3.3]{cc}) that $\langle \gamma,\beta \rangle - \langle \beta,\gamma \rangle > 0$ for every $\gamma$ with $\Gr_\gamma(S_\beta^s) \neq \emptyset$, unless $\gamma$ is a multiple of $\beta$.
	
	We write $S_\beta^s$ as $S_\beta \otimes k^s$ and consider the closed embedding $\Gr_t(k^s) \to \Gr_{t\beta}(S_\beta \otimes k^s)$ which is given by mapping a $t$-dimensional subspace $U \sub k^s$ to the subrepresentation $S_\beta \otimes U$. In order to prove the lemma, it suffices to show that every $t\beta$-dimensional subrepresentation $N$ of $S_\beta \otimes k^s$ is of this form. The proof proceeds by induction on $s$. The basis $s=1$ is obvious. So let us assume $s > 1$ and consider the morphism in (\ref{e:Gr}) for $\gamma = t\beta$. As
	$$
		0 = \langle t\beta,\beta \rangle - \langle \beta,t\beta \rangle = \langle \delta,\beta \rangle - \langle \beta,\delta \rangle + \langle \epsilon,\beta \rangle - \langle \beta,\epsilon \rangle
	$$
	we deduce that both $\delta$ and $\epsilon$ must be multiples of $\beta$. We distinguish two cases. In case $\delta = 0$, we get $i^{-1}(N) = 0$ and $\pi(N) \in \Gr_{t\beta}(S_\beta \otimes k^{s-1})$. In the above context, $i: S_\beta \to S_\beta \otimes k^s$ is given by the first coordinate vector and $\pi: S_\beta \otimes k^s \to S_\beta \otimes k^{s-1}$ is the projection along the first coordinate. By induction assumption, $\pi(N) = S_\beta \otimes U$ for some $t$-dimensional subspace of $k^{s-1}$. The other case is $\delta = \beta$. That means $i^{-1}(N) = S_\beta$ and $\pi(N) \in \Gr_{(t-1)\beta}(S_\beta \otimes k^{s-1})$; hence $\pi(N) = S_\beta \otimes U'$ for some subspace $U' \sub k^{s-1}$ of dimension $t-1$. Setting $U = \langle e_1 \rangle \oplus U'$, we arrive at $N = S_\beta \otimes U$.
\end{proof}

We fix a non-Schurian root $\alpha$ with canonical exceptional decomposition $\alpha=\alpha_1^{d_1}+\alpha_2^{d_2}+\alpha_3^{d_3}$. For the remaining part of this section define
\[l:=\hom(\alpha_2,\alpha_1),\,m:=\ext(\alpha_3,\alpha_1),\,n:=\ext(\alpha_3,\alpha_2).\]
\begin{rem}\label{beml}
We have $l>0$ because otherwise we would have $\hom(\alpha_2,\alpha_1)=\hom(\alpha_3,\alpha_1)=0$, see Lemma \ref{kipp1}. This would yield a contradiction. Indeed, since $\hom(\alpha_3,\alpha_2)=0$, it would follow that $S_{\alpha_1},S_{\alpha_2},S_{\alpha_3}$ are simple in $\mathcal C(S_{\alpha_1},S_{\alpha_2},S_{\alpha_3})$, see \cite[Lemma 2.35]{dw2}. But since $\mathcal C(S_{\alpha_1},S_{\alpha_2},S_{\alpha_3})=\Rep(Q(\un m))$ by \cite[Lemma 3]{cb}, this contradicts the assumption that $Q(\un m)$ is connected.
\end{rem} 
\begin{lem}\label{grlem} Let $(\phi_1,\ldots,\phi_l)$ be a basis of $\Hom(S_{\alpha_2},S_{\alpha_1})$. For $t\leq s$, there exists an embedding $$\Gr_{t}(k^s)\cong\Gr_{t\alpha_1}(S_{\alpha_1}^s)\hookrightarrow\Gr_{t(l\alpha_2)}((S_{\alpha_2}^l)^s) \cong \Gr_{lt}(k^{ls})$$ induced by commutative diagrams of the form 
\[
\begin{xy}\xymatrix@R14pt@C25pt{(S_{\alpha_2}^l)^t\ar@{^{(}->}[rr]\ar[d]^{\phi^t}&&(S_{\alpha_2}^l)^s\ar[d]^{\phi^s}\\S_{\alpha_1}^t\ar@{^{(}->}[rr]&&S_{\alpha_1}^s}
\end{xy}
\]
where $\phi=(\phi_1,\ldots,\phi_l)$.
\end{lem}
\begin{proof}
Since $(\alpha_1,\alpha_2)$ is an exceptional sequence such that $\ext(\alpha_2,\alpha_1)=0$ and $l=\hom(\alpha_2,\alpha_1)>0$, the basis $(\phi_1,\ldots,\phi_l)$ of $\Hom(S_{\alpha_2},S_{\alpha_1})$ induces a morphism $\phi:S_{\alpha_2}^l\to S_{\alpha_1}$ and an exceptional representation $S_{\gamma}\in S_{\alpha_2}^{\perp}$ which is either the kernel or cokernel of $\phi$. By construction, we have $\Hom(S_{\alpha_2},S_{\alpha_1})\cong\Hom(S_{\alpha_2},S_{\alpha_2}^l)$. 

If $S_{\gamma}$ is the cokernel, we have $\Hom(S_{\gamma},S_{\alpha_1})=0$ because the endomorphism ring of $S_{\alpha_1}$ is trivial and $S_\gamma$ is a proper factor of $S_{\alpha_1}$. In both cases, we thus get an injective map $$\Hom(S_{\alpha_1}^t,S_{\alpha_1}^s)\hookrightarrow \Hom((S_{\alpha_2}^l)^t,S_{\alpha_1}^s)\cong\Hom((S_{\alpha_2}^l)^t,(S_{\alpha_2}^l)^s)$$
which induces both the diagrams and the embedding. 
\end{proof}
\begin{rem} We use the notation from Lemma \ref{grlem}.
\begin{enumerate} 
\item If $S_\gamma$ is the kernel, note that we can actually not apply Ringel's reflection functor because $S_{\alpha_1}$ is the quotient of a direct sum of copies of $S_{\alpha_2}$'s.
\item Considering the Auslander-Reiten quiver of a generalized Kronecker quiver, it can actually be seen that $S_{\gamma}$ is the cokernel if and only if $S_{\gamma}$ is simple and injective in the category $\mathcal C:=\mathcal C(S_{\alpha_1},S_{\alpha_2})$ which means $\Hom(S_{\gamma},\underline{\quad})\mid_{\mathcal C}=0$. 

\item Since $\Hom(S_{\alpha_i},S_{\alpha_i})=k$ for $i=1,2$, we can make the injective map from Lemma \ref{grlem} explicit using matrices. Thus, in particular, properties like injectivity are preserved.
\end{enumerate}
\end{rem}

For a fixed $k$-vector space $V$ of dimension $n$ and natural numbers $1\leq d_1<\ldots<d_k\leq n$ with $k\leq n$, we define the corresponding (partial) flag variety by
\[\mathrm{Fl}_{(d_1,\ldots,d_k)}(V):=\{U_1\subset\ldots\subset U_k\subset V\mid \dim U_i=d_i\}.\]
Setting $d_{k+1}=n$ recall that 
\[\dim \mathrm{Fl}_{(d_1,\ldots,d_k)}(V)=\sum_{i=2}^{k+1}\dim \Gr_{d_{i-1}}(k^{d_i})=\sum_{i=2}^{k+1}d_{i-1}(d_i-d_{i-1}).\]

Define $V:=\Hom(S_{\alpha_2},S_{\alpha_2}^{ls})\cong k^{ls}$. Fix natural numbers $r,s,t$ with $t\leq s$. Then every $r$-dimensional subspace, i.e. every point of the usual Grassmannian $\Gr_r(V)$, defines an injection $S_{\alpha_2}^r\hookrightarrow S_{\alpha_2}^{ls}$ and vice versa. The next aim is to define two subvarieties of $\Gr_r(V)$ such that their intersection consists of morphisms which give rise to a commutative diagram of the form (\ref{diag}). Thus on the one hand, we are interested in morphisms $h:S_{\alpha_2}^r\to S_{\alpha_1}^s$ of maximal possible rank factoring through $S_{\alpha_1}^t$. On the other hand, they should also factor through $\smash{S_\delta^{d_3}}$. By Lemma \ref{grlem}, morphisms satisfying the first condition can be described by certain morphisms $\hat h: S_{\alpha_2}^r\to (S_{\alpha_2}^{l})^s$ factoring through $(S_{\alpha_2}^{l})^t$. As we will see, these morphisms can be described in terms of subvarieties of the Grassmannian $\Gr_r(V)$. 

\begin{rem}\label{upper}The considerations from subsections \ref{haupt} and \ref{subvarieties} suggest to set 
\begin{align*}
	r &= nd_3-d_2, \\
	s &= (nl-m)d_3, \text{ and} \\
	t &= (nl-m)d_3-d_1-\Sc{\hat{\alpha}}{\alpha_1}.
\end{align*}
Indeed, then the cokernel of the upper row of (\ref{diag}) is of dimension $\hat\alpha$ and $s-t=d_1+\Sc{\hat\alpha}{\alpha_1}$. 
\end{rem}

\begin{defi}\label{typeone}
We say that a non-Schurian root $\alpha$ is of type one if $r\leq lt$ and $r\leq nd_3\leq ls$ where $r,s$ and $t$ take the values of Remark \ref{upper}.
\end{defi}

\begin{rem}\label{bem134}
Note that the second inequality is equivalent to $ld_1\geq d_2$. Moreover, we have $r=nd_3-d_2\geq 1$ because $(d_3,d_2)$ is an imaginary Schur root of $K(n)$. Since $l>0$ by Remark \ref{beml}, for roots of type one, we thus have $s,t\geq 1$. 
It seems that the large majority of non-Schurian roots is of type one. The non-Schurian roots which are not of type one lead to special cases treated in subsection \ref{special}.  From now on we assume that $\alpha$ is of type one and that $r,s$ and $t$ take the values of Remark \ref{upper}. But actually, the construction presented in this subsection can be generalized to other values of $r,s$ and $t$ for which $\Gr_r(V)$ and the two subvarieties can be defined.
\end{rem}
\begin{rem} \label{generalization}
The results of Theorem \ref{mainresults} can be generalized straightforwardly to non-Schurian roots $\alpha$ of type one which satisfy Lemma \ref{kipp1} and whose canonical 
decomposition is of the form $\alpha=\alpha_1^{d_1}\oplus\hat\alpha$ (resp. $\alpha=\hat\alpha\oplus\alpha_1^{d_1}$) where $\hat\alpha=\alpha_2^{d_2}+\alpha_3^{d_3}$ and, moreover, $\alpha_2,\alpha_3$ are simple in $\alpha_1^{\perp}$ (resp. ${^\perp}\alpha_1$). As the proof remain the same, we restrict to the case of quivers with three vertices for convenience.
\end{rem} 

We have $V=\Hom(S_{\alpha_2},S_{\alpha_2}^{ls})$ and let $V_0=\Hom(S_{\alpha_1},S_{\alpha_1}^s)$. We write the commutative diagram from Lemma \ref{grlem} (with $t=1$) as
\[
\begin{xy}\xymatrix@R14pt@C25pt{S_{\alpha_2} \otimes k^l \ar@{^{(}->}[rr]^{f'} \ar[d]^{\phi}&&S_{\alpha_2} \otimes k^l \otimes k^s \ar[d]^{\phi \otimes \operatorname{id}_{k^s}}\\S_{\alpha_1}\ar@{^{(}->}[rr]^{f} &&S_{\alpha_1} \otimes k^s.}
\end{xy}
\]
The morphism $f$ must be of the form $f = \operatorname{id} \otimes v$ for a vector $v \in k^s$. Choosing $f' = \operatorname{id}_{S_{\alpha_2}} \otimes \operatorname{id}_{k^l} \otimes v$ makes the diagram commutative. Identifying $\smash{k^l} \otimes k^s \cong \smash{(k^s)^l} \cong \smash{k^{ls}}$, we have shown that the embedding
$$
	k^s \cong \Hom(S_{\alpha_1},S_{\alpha_1}^s) = V_0 \hookrightarrow V = \Hom(S_{\alpha_2},S_{\alpha_2}^{ls}) \cong k^{ls}
$$
is given by $v \mapsto v \oplus \ldots \oplus v$ (the $l$-fold direct sum) -- regardless of the choice of $\phi$. This shows that under the identification $V \cong V_0^l$, the closed embedding $\Delta:\Gr_t(V_0)\to\Gr_{lt}(V)$ from Lemma \ref{grlem} is given by $U\mapsto U^l$.

The flag variety $\Fl_{(r,lt)}(V)$ comes equipped with projections
$$
	\Gr_r(V) \xot{}{\psi_1} \Fl_{(r,lt)}(V) \xto{}{\psi_2} \Gr_{lt}(V).
$$
Every point in the image of $\psi_1$ defines a morphism $S_{\alpha_2}^r\to (S_{\alpha_2}^{l})^s$ factoring through $(S_{\alpha_2}^{l})^t$. But since we are only interested in morphisms $h:S_{\alpha_2}^r\to S_{\alpha_1}^s$ factoring through $S_{\alpha_1}^t$, we need to restrict $\psi_1$ to the subvariety 
\[Y:= \psi_2^{-1} \Delta(\Gr_t(V_0)) = \{(U_1,U_2)\in \Fl_{(r,lt)}(V)\mid U_2\in \Delta(\Gr_t(V_0))\}.\] 
We denote the subvariety $\psi_1(Y)$ of $\Gr_r(V)$ by $X_1^\alpha$. The variety is constructed in such a way that the following holds:

\begin{lem}
	Every point $p$ of $X_1^\alpha$ corresponds to a unique morphism $f_p: S_{\alpha_2}^r \to S_{\alpha_1}^s$ for which there exist $f_1:S_{\alpha_2}^r\to S_{\alpha_1}^t$ and $i_2:S_{\alpha_1}^t\to S_{\alpha_1}^s$ with $f_p=i_2\circ f_1$. 
\end{lem}

We show in subsection \ref{dim} that
$$\dim X^{\alpha}_1=\dim \Gr_{r}(k^{ls})+\dim \Gr_t(k^{s})=r(ls-r)+t(s-t).$$

In order to define the second subvariety, we consider the following exact sequence induced by a basis of $\Ext(S_{\alpha_3},S_{\alpha_2})$:
\[\ses{S_{\alpha_2}^{nd_3}}{S^{d_3}_{\delta}}{S^{d_3}_{\alpha_3}}.\]
Since $\Sc{\hat\alpha}{\alpha_1}\geq 0$ and since $\hat\alpha$ corresponds to a root of $K(n)$, we have $m\leq n\cdot l$ and thus $\ext(\delta,\alpha_1)=0$. Since $(\alpha_1,\delta)$ is an exceptional sequence, we get a morphism $S_{\delta}\to S_{\alpha_1}^{\hom(\delta,\alpha_1)}$ induced by a basis of $\Hom(S_{\delta},S_{\alpha_1})$ which induces a linear map $$\Hom(S_{\alpha_2},S^{d_3}_{\delta})\to\Hom(S_{\alpha_2},S_{\alpha_1}^s)\cong\Hom(S_{\alpha_2},(S_{\alpha_2}^l)^s)=V$$
where $s=d_3\cdot\hom(\delta,\alpha_1)\geq 1$. Let $f_2:S_\delta^{d_3}\to S_{\alpha_1}^s$ be the induced diagonal morphism. This means that every homomorphism contained in the subspace of $V$ defined by $\Hom(S_{\alpha_2},S^{d_3}_{\delta})$ defines a homomorphism in $\Hom(S_{\alpha_2},(S_{\alpha_2}^l)^s)$ factoring through $f_2$. Furthermore, every $r$-dimensional subspace contained in $W:=\Hom(S_{\alpha_2},S_{\delta}^{d_3})$ defines a point of the Grassmannian $\Gr_r(V)$ corresponding to an injection of $S_{\alpha_2}^r\to (S_{\alpha_2}^l)^s$ factoring through $S^{d_3}_{\delta}$. In turn, such morphisms correspond to points of the subvariety 
\[X^{\alpha}_2=\{U\in \Gr_r(V)\mid U\subset W\}\cong\Gr_{r}(W).\]
We summarize:

\begin{lem}
	A point $p$ of the variety $X_2^\alpha$ corresponds to a unique morphism $f_p: S_{\alpha_2}^r \to S_{\alpha_1}^s$ factoring as $S_{\alpha_2}^r \hookrightarrow \smash{S_\delta^{d_3}} \xto{}{f_2} S_{\alpha_1}^s$. This factorization is unique.
\end{lem}

It is well-known that $\dim X^{\alpha}_2=r(\dim W-r)$, see also subsection \ref{dim} for more details. 
Denoting the intersection of $X_1^\alpha$ and $X_2^\alpha$ inside $\Gr_r(V)$ with $I^\alpha$, we deduce from the two previous lemmas:

\begin{thm}\label{intersection} Let $\alpha$ be a non-Schurian root of $Q(\un{m})$ which is of type one. Every morphism $f_p$ induced by a point $p\in I^{\alpha}$ gives rise to a commutative diagram 
 \[
\begin{xy}\xymatrix@R30pt@C29pt{0\ar[r]&S_{\alpha_2}^r\ar@{-->}[rd]^{f_p}\ar[r]^{i_1}\ar[d]^{f_1}&S_{\delta}^{d_3}\ar[d]^{f_2}\ar[r]^{\pi_1}&M_{\hat\alpha}\ar[r]\ar[d]^{f_3}&0\\0\ar[r]&S_{\alpha_1}^t\ar[r]^{i_2}&S_{\alpha_1}^s\ar[r]^{\pi_2}&S_{\alpha_1}^{s-t}\ar[r]&0}
\end{xy}
\]
such that $f_3$ is surjective. Moreover, if $d_3=1$ the cokernel of $i_1$ is indecomposable.
\end{thm}
\begin{proof}
On the one hand every $p\in I^{\alpha}$ yields a morphism $f_p$ factoring through $S_{\delta}^{d_3}$ and on the other hand $f_p$ factors through $S_{\alpha_1}^t$. Since the left square commutes, we have $(\pi_2\circ f_2)\circ i_1=0$ and the universal property of the cokernel yields a morphism $f_3$ as in the diagram such that everything commutes. 
Since $f_2$ is induced by a basis of $\Hom(S_{\delta}^{d_3},S_{\alpha_1})$, it follows that $\pi_2\circ f_2$ is of maximal rank, i.e. surjective because $\pi_2$ is surjective. Thus $f_3$ is forced to be surjective. 

If $d_3=1$, the indecomposability of $\mathrm{coker}(i_1)$ already follows by the injectivity of $i_1$, see Remark \ref{rem1}.\end{proof}

\begin{rem}\label{proj} 
We have that $P:=S_{\delta}\oplus S_{\alpha_2}$ is a partial tilting module. Moreover, $\End(P)$ is isomorphic to the path algebra of $K(\hom(\alpha_2,\delta))$ where $\hom(\alpha_2,\delta)=\ext(\alpha_3,\alpha_2)$. This implies that the representations $M_{\hat\alpha}$ obtained as the cokernel of an exact sequence of the form
\[\ses{S_{\alpha_2}^{r}}{S_{\delta}^{d_3}}{M_{\hat\alpha}}\]
are in one-to-one correspondence to representations $X$ of $K(n)$ of dimension $d:=(r,d_3)$ such that $\Hom(X,S_{q_1})=0$. Here $S_{q_1}$ denotes the simple representation corresponding to $q_1\in K(n)_0$. Furthermore, $S_{\alpha_2}$ and $S_\delta$ are the indecomposable projective representations in $S_{\alpha_1}^{\perp}$. In particular, the exact sequence yields a minimal projective resolution of $M_{\hat\alpha}$ in $S_{\alpha_1}^\perp$. 
 Now the natural group action of $\mathrm{Gl}_{r}(k)\times \mathrm{Gl}_{d_3}(k)$ on $R_{d}(K(n))$ corresponds to diagrams
 \[
\begin{xy}\xymatrix@R20pt@C29pt{0\ar[r]&S_{\alpha_2}^r\ar[r]^{i_1}\ar[d]^{g_1}&S_{\delta}^{d_3}\ar[d]^{g_2}\ar[r]^{\pi_1}&M_{\hat\alpha}\ar[r]\ar[d]^{g_3}&0\\0\ar[r]&S_{\alpha_2}^r\ar[r]^{i_2}&S_{\delta}^{d_3}\ar[r]^{\pi_2}&M'_{\hat\alpha}\ar[r]&0}
\end{xy}                               
\]
where the maps $g_i$ are isomorphisms. It is straightforward that, on the Grassmannian side, the $\mathrm{Gl}_{r}(k)$-action corresponds to the usual base change action. Thus if we want to classify representations in $I^{\alpha}$ up to isomorphism, we only need to consider the $\mathrm{Gl}_{d_3}(k)$-action.
\end{rem}

\begin{rem}\label{kacgen}

In the next subsection, we calculate the dimension of $I^{\alpha}$ which turns out to be at least $d_3^2-\Sc{\alpha}{\alpha}$. Thus taking into account the $\mathrm{Gl}_{d_3}(k)$-action -- note that the diagonally embedded $k^\ast$ acts trivially -- there exists at least  a $(1-\Sc{\alpha}{\alpha})$-parameter family of isomorphism classes of representations in $I^{\alpha}$. By Kac's Theorem \cite[Theorem C]{kac2} this is also an upper bound if all representations in $I^{\alpha}$ are indecomposable. The same is true if one representation in $I^\alpha$, and thus an open subset of representations in $I^\alpha$, is Schurian, see also Lemma \ref{genhom}.
\end{rem}

The following corollary establishes the connection to Ringel's reflection functor recalled in section \ref{ringelrefl}:
\begin{kor}\label{Ringelrefl}Let $t$ and $r$ be defined as in Remark \ref{upper}. Then the points of $I^{\alpha}$ correspond precisely to those representations $M_{\hat\alpha}$ which can be written as the cokernel of short exact sequences
$$\ses{S_{\alpha_1}^{d_1}}{M_\alpha}{M_{\hat\alpha}}$$
with $\Hom(S_{\alpha_1},M_\alpha)=d_1$ and such that $M_\alpha$ has no direct summand which is 
 isomorphic to $S_{\alpha_1}$ or $S_{\alpha_2}$.

\end{kor}
\begin{proof}By construction, for every representation $M_{\hat\alpha}$ corresponding to a point of $I^{\alpha}$ we have $M_{\hat\alpha}\in S_{\alpha_1}^{\perp}$ and $\dim\Ext(M_{\hat\alpha},S_{\alpha_1})\geq d_1$. Moreover, $M_{\hat\alpha}$ has no direct summand which is isomorphic to $S_{\alpha_1}$ or $S_{\alpha_2}$. Thus the middle terms of the induced short exact sequences $\ses{S_{\alpha_1}^{d_1}}{M_\alpha}{M_{\hat\alpha}}$ satisfy the claimed properties. 

Conversely let $M_{\alpha}$ be of dimension $\alpha$ such that $\dim\Hom(S_{\alpha_1},M_\alpha)=d_1$ and such that $S_{\alpha_1}$ and $S_{\alpha_2}$ are no direct summands of $M_\alpha$. Then we have $\Ext(S_{\alpha_1},M_\alpha)=0$ and, by Theorem \ref{ringel}, there exists a short exact sequence $\ses{S_{\alpha_1}^{d_1}}{M_\alpha}{M_{\hat\alpha}}$ such that $M_{\hat\alpha}\in S_{\alpha_1}^\perp$ and $\dim\Ext(M_{\hat\alpha},S_{\alpha_1})\geq d_1$. It follows that we have $\dim\Hom(M_{\hat\alpha},S_{\alpha_1})\geq\Sc{\hat\alpha}{\alpha_1}+d_1$. Since $M_{\alpha}$ has no direct summand isomorphic to $S_{\alpha_2}$, the same is true for $M_{\hat\alpha}$ because $\Ext(S_{\alpha_2},S_{\alpha_1})=0$.
In particular, $M_{\hat\alpha}$ fits into a commutative diagram as in Theorem \ref{intersection}.
\end{proof}

\subsection{Dimensions}\label{dim}
\noindent  
Again let $\alpha=\alpha_1^{d_1}+\alpha_2^{d_2}+\alpha_3^{d_3}$ be the canonical exceptional decomposition of a root $\alpha$ of type one with
\[\ext(\alpha_3,\alpha_2)=n,\,\hom(\alpha_2,\alpha_1)=l,\,\ext(\alpha_3,\alpha_1)=m.\]
	
Then Kac's result yields that the isomorphism classes of indecomposables of dimension $\alpha$ can be described by
$$1-\Sc{\alpha}{\alpha}=1-\sum_{i=1}^3d_i^2-ld_1d_2+md_1d_3+nd_2d_3$$
parameters. 
As in the previous subsection, let
\begin{align*}
	r &= nd_3-d_2, \\
	s &= (nl-m)d_3, \text{ and} \\
	t &= (nl-m)d_3-d_1-\Sc{\hat{\alpha}}{\alpha_1}.
\end{align*}
Moreover, we have defined the vector spaces $V_0$, $V = V_0^l$, and $W \sub V$ in the previous subsection. Their dimensions are
$$
	\dim V_0 = s,\ \dim V = ls,\ \text{and } \dim W = nd_3 =: w. 
$$
We abbreviate $X_i:=X_i^{\alpha}$ (for $i = 1,2$) and $I :=I^{\alpha}$ in this case. Then, we get $X_2 = \Gr_r(W)$ and
$X_1 = \{ U \in \Gr_r(V) \mid \text{ex. } U_0 \in \Gr_t(V_0) \text{ with } U \sub U_0^l \}$.

\begin{thm} \label{lem_dim_int}
	If $X_1$ and $X_2$ intersect then every irreducible component of $X_1 \cap X_2$ has dimension at least $d_3^2 - \langle \alpha,\alpha \rangle$.
\end{thm}

The rest of this subsection deals with the proof of Theorem \ref{lem_dim_int}. We need some auxiliary results.

We introduce the following construction: Let $\pr_i: V = V_0^l \to V_0$ be the projection to the $i$\textsuperscript{th} factor. For a subspace $U \sub V$, we define $\pr(U) \sub V_0$ to be the sum over the images of $U$ under these projections, i.e. 
$$
	\pr(U) = \pr_1(U) + \ldots + \pr_l(U).
$$
It is easy, yet crucial, to observe the following:

\begin{lem} \label{lem_pr}
	For subspaces $U \sub V$ and $U_0 \sub V_0$, we have $U \sub U_0^l$ if and only if $\pr(U) \sub U_0$.
\end{lem}

\begin{proof}
	Suppose that $U \sub U_0^l$. Every $u \in U$ decomposes as $u = \sum_i \pr_i(u)$ and by assumption $\pr_i(u) \in U_0$. Therefore $\pr_i(U) \sub U_0$ and thus, $\pr(U) \sub U_0$. Conversely, we assume that $\pr(U) \sub U_0$ and take $u \in U$. Then $\pr_i(u) \in \pr_i(U) \sub \pr(U) \sub U_0$, whence $u = \sum_i \pr_i(u) \in U_0^l$.
\end{proof}

In order to compute the dimension of $X_1$, we consider $Y = \psi_2^{-1} \Delta (\Gr_t(V_0))$ (as defined in the previous subsection) which equals the set of all flags of the form $U \sub \smash{U_0^l}$ with $U \in \Gr_r(V)$ and $U_0 \in \Gr_t(V_0)$. Let $\EE \to \Gr_{lt}(V)$ be the universal rank $lt$-bundle. As a variety over $\Gr_{lt}(V)$, the flag variety $\smash{\Fl_{(r,lt)}(V)}$ identifies with the Grassmannian $\Gr_r(\EE)$. Thus, $\psi_2$ is locally trivial and its fiber is a Grassmannian $\Gr_r(k^{lt})$. Therefore, $Y$ is irreducible of dimension
$$
	\dim Y = \dim \Gr_t(V_0) + \dim \Gr_r(k^{lt}) = t(s-t) + r(lt - r).
$$
The following lemma will ensure that the dimension of $X_1 = \psi_1(Y)$ coincides with the dimension of $Y$.
\begin{lem} \label{lem_tlr}
We have $1 \leq t=lr-d_1\leq lr$.
\end{lem}
\begin{proof}
By Remark \ref{bem134}, we have $t\geq 1$. Moreover, $nl-m=\hom(\delta,\alpha_1)=\Sc{\delta}{\alpha_1}$ which yields
\[t=\Sc{\delta^{d_3}}{\alpha_1}-\Sc{\hat{\alpha}}{\alpha_1}-d_1.\]
Since $l=\Sc{\alpha_2}{\alpha_1}$ and $m=\ext(\alpha_3,\alpha_1)$, we get
\[lr = l(nd_3-d_2)=\Sc{\delta^{d_3}}{\alpha_1}-\Sc{\alpha_2^{d_2}}{\alpha_1}-\Sc{\alpha_3^{d_3}}{\alpha_1}=\Sc{\delta^{d_3}}{\alpha_1}-\Sc{\hat{\alpha}}{\alpha_1}.\]
Thus the claim is equivalent to $d_1\geq 0$. \end{proof}

\begin{pro} \label{prop_birat}
	The morphism $\psi_1$ restricts to a birational morphism $Y \to X_1$.
\end{pro}

\begin{proof}
	We consider the restriction of the map $\psi_1: \Fl_{(r,lt)}(V) \to \Gr_r(V)$ which gives a surjective morphism $Y \to X_1$. Let $U \in X_1$ and consider the fiber $Y_U$. We obtain
	\begin{align*}
		Y_U &\cong \{ U_0 \in \Gr_t(V_0) \mid U \sub U_0^l \} \\
			&= \{U_0 \in \Gr_t(V_0) \mid \pr(U) \sub U_0 \} \\
			&\cong \Gr_{t-k_U}(V_0/\pr(U))
	\end{align*}
	using Lemma \ref{lem_pr}. Here, $k_U$ is defined as $\dim \pr(U)$. As $U \in X_1$, there exists $U_0 \in \Gr_t(V_0)$ such that $U \sub U_0^l$. Thus $\dim \pr(U) \leq t$. We show that
	$$
		t = \max\{ k_U \mid U \in X_1 \}.
	$$
	Choose a basis $v^1,\ldots,v^s$ of $V_0$ and let $\{ \smash{v_i^j} \}$ be the basis of $V = V_0^l$ where $\smash{v_i^j}$ is the vector $v^j$ located in the $i$\textsuperscript{th} copy of $V_0$. Put $U_0$ to be the span of $v^1,\ldots,v^t$. Choose a natural number $q$ and $k \in \{1,\ldots,l\}$ with $t = (q-1)l+k$ and define
	$$
		u_1 = v_1^1 + \ldots + v_l^l,\ u_2 = v_1^{l+1} + \ldots + v_l^{2l},\ \ldots,\ u_q = v_1^{(q-1)l+1} + \ldots + v_k^t.
	$$
	This choice assures that the vectors $\pr_i(u^j)$ are linearly independent. As $t \leq lr$, we have $r \geq q$. Any $r$-dimensional subspace $U$ of $\smash{U_0^l}$ containing $u_1,\ldots,u_q$ fulfills $\dim \pr(U) = t$. Choosing such a subspace $U$, the fiber $Y_U$ is a singleton. 
ls $\dim \pr(U) = t$. Choosing such a subspace $U$, the fiber $Y_U$ is a singleton. 
	The association $U \mapsto \pr(U)$ gives a morphism $X_1^o \to Y$ on the dense open subset of all $U \in X_1$ for which $\dim \pr(U) = t$ (whose image we denote $Y^o$) and provides an inverse to $\psi_1$ restricted to $Y^o \to X_1^o$.
\end{proof}

\begin{kor} \label{cor_dim}
	The dimension of $X_1$ is also $t(s-t) + r(lt - r)$.
\end{kor}

\begin{proof}[\textit{Proof of Theorem \ref{lem_dim_int}}]
Obviously $\dim X_2 = r(w-r)$. Moreover, $\Gr_r(V)$ is a non-singular variety whence the diagonal embedding $\Gr_r(V) \to \Gr_r(V) \times \Gr_r(V)$ is a regular embedding of codimension $\dim \Gr_r(V) = r(ls-r)$. Using \cite[Lemma 7.1]{ful}, we deduce from the fiber square
$$
	\begin{xy}
		\xymatrix{
		X_1 \cap X_2 \ar[r] \ar[d]	&   X_1 \times X_2 \ar[d]  \\
		\Gr_r(V) \ar[r]			&   \Gr_r(V) \times \Gr_r(V)   }
	\end{xy}
$$ 
that every irreducible component of the intersection $X_1 \cap X_2$ has dimension at least
\begin{align*}
	\dim X_1 + \dim X_2 - \dim \Gr_r(V) &= t(s-t) + r(lt-r) + r(w-r) - r(ls-r) \\
		&= r(w-r) - (lr-t)(s-t)
\end{align*}
and by a straightforward calculation, we see that $r(w-r)-(lr-t)(s-t)$ equals $d_3^2-\langle \alpha,\alpha \rangle$.
\end{proof}

\subsection{Intersecting subvarieties of Grassmannians}\label{intsec}
\noindent Our next task is to prove that the intersection $I = X_1 \cap X_2$ is non-empty. We do not know an elementary proof for this. The strategy for proving this result is to show that the intersection product $[X_1] \cdot [X_2]$ in the Chow ring $A^*(\Gr_r(V))$ is non-zero which implies, by the existence of refined intersections (cf.\ \cite[Section 8.1]{ful}), that $X_1 \cap X_2 \neq \emptyset$.

Again, we consider the variety $Y$ and regard it as the subvariety of $\Gr_r(V) \times \Gr_t(V_0)$ of pairs $(U,U_0)$ with $U \sub U_0^l$. The image of $Y$ under the projection to the first component equals $X_1$. 
Let $\UU$ be the vector bundle on $\Gr_r(V) \times \Gr_t(V_0)$ which arises as the pull-back of the universal rank $r$-subbundle of the trivial bundle on $\Gr_r(V)$ with fiber $V$ and let $\QQ_0$ be the pull-back to $\Gr_r(V) \times \Gr_t(V_0)$ of the universal rank $(s-t)$-quotient bundle on $\Gr_t(V_0)$. As a closed subset, $Y$ equals the vanishing set $Z(\Psi)$, where
$$
	\Psi: \UU \to \pi^*V_0^l \to \QQ_0^l
$$
is interpreted as a global section of the bundle $\UU^\vee \otimes \QQ_0^l$. In the above context, $\pi: \Gr_r(V) \times \Gr_t(V_0) \to \Spec k$ is the structure map. By Corollary \ref{cor_dim}, the codimension of $Y$ in $\Gr_r(V) \times \Gr_t(V_0)$ is precisely $lr(s-t) = \rk(\UU^\vee \otimes \QQ_0^l)$. As the ambient variety $\Gr_r(V) \times \Gr_t(V_0)$ is non-singular, we deduce that $\Psi$ is a regular section and that $[Z(\Psi)] = \mathbb{Z}(\Psi)$ (see \cite[Example 14.1.1]{ful}). By \cite[Proposition 14.1]{ful}, the image of $\Zn(\Psi)$ in $A_*(\Gr_r(V) \times \Gr_t(V_0))$ equals
$$
	c_{lr(s-t)}(\UU^\vee \otimes \QQ_0^l).
$$
Let $x_1,\ldots,x_r$ be the Chern roots of $\UU^\vee$ and let $y_1,\ldots,y_{s-t}$ be the Chern roots of $\QQ_0$. The total Chern class of $\UU^\vee \otimes \QQ_0^l$ is $\prod_{i=1}^r \prod_{j=1}^{s-t} (1+x_i+y_j)^l$. Therefore, the top Chern class of this bundle is
$$
	c_{lr(s-t)}(\UU^\vee \otimes \QQ_0^l) = \prod_{i=1}^r \prod_{j=1}^{s-t} (x_i+y_j)^l.
$$
Proposition \ref{prop_birat} asserts that the push-forward $\pi_{1,*} [Y]$ under the projection $\pi_1: \Gr_r(V) \times \Gr_t(V_0) \to \Gr_r(V)$ agrees with $[X_1] \in A_*(\Gr_r(V))$ (the cycle associated with the subvariety $X_1$). By Proposition \ref{basis_thm}, $A_*(\Gr_t(V_0))$ is free with the basis elements $\Delta_\lambda = \det(c_{\lambda_i+j-i}(\QQ_0))$ where $\lambda$ ranges over all partitions fitting into a $t \times (s-t)$-box. 

\begin{lem}
	The push-forward $\pi_{1,*} \big(c_{lr(s-t)}(\UU^\vee \otimes \QQ_0^l)\big)$ equals
	$$
		\left( \sum_{1 \leq i_1 < \ldots < i_{lr-t} \leq lr} x_{\lceil i_1/l \rceil} \ldots x_{\lceil i_{lr-t}/l \rceil} \right)^{s-t}.
	$$
\end{lem}

\begin{proof}
	The push-forward map $\pi_{1,*}$ sends $\smash{\Delta_{(s-t)^t}}$ to $1$ and the rest of the basis elements to $0$. It therefore suffices to show that the coefficient of $\smash{\Delta_{(s-t)^t}} = c_{s-t}(\QQ_0)^t = y_1^t \ldots y_{s-t}^t$ in the top Chern class $\smash{c_{lr(s-t)}(\UU^\vee \otimes \QQ_0^l) = \prod_{i=1}^r \prod_{j=1}^{s-t} (x_i+y_j)^l}$ is the desired expression. This coefficient is $\smash{f^{s-t}}$, where $f = f(x_1,\ldots,x_r)$ denotes the coefficient of $y^t$ in the expression
	$$
		\prod_{i=1}^r (x_i+y)^l.
	$$
	Then $f$ is the sum over all monomials in $x_1,\ldots,x_r$ of degree $lr-t$ in which every $x_i$ occurs with a power of at most $l$. Such a monomial can be written as $x_{\alpha_1}\ldots x_{\alpha_{lr-r}}$ for a unique non-decreasing sequence $1 \leq \alpha_1 \leq \ldots \leq \alpha_{lr-t} \leq r$ for which no number $i \in \{1,\ldots,r\}$ occurs more than $l$ times. These sequences are in bijection with increasing sequences $1 \leq i_1 < \ldots < i_{lr-t} \leq lr$ by mapping $i_\nu$ to $\lceil i_\nu/l \rceil$.
\end{proof}
%
We abbreviate $p = lr-t$. In order to display the parenthesized expression in the previous lemma as a linear combination of monomial symmetric functions evaluated at the $x_i$'s, we prove the following identity of symmetric functions:

\begin{lem}
	In the ring of symmetric functions in the variables $x_1,x_2,\ldots$, we have
	$$
		\sum_{i_1 < \ldots < i_p} x_{\lceil i_1/l \rceil} \ldots x_{\lceil i_p/l \rceil} = \sum_\lambda \left( \prod_{i \geq 1} \binom{l}{\lambda_i} \right) m_\lambda,
	$$
	where the sum ranges over all partitions $\lambda$ of $p$ with $l \geq \lambda_1$.
\end{lem}

\begin{proof}
	Using that $f = \smash{\sum_{i_1 < \ldots < i_p} x_{\lceil i_1/l \rceil} \ldots x_{\lceil i_p/l \rceil}}$ is a non-negative integral linear combination of monomials, we see that $f$ can be displayed as $\sum_\lambda a_\lambda m_\lambda$ with $a_\lambda \in \Zn_{\geq 0}$. For a monomial $\smash{x^\lambda}$ (corresponding to a partition $\lambda$) occurring in $f$, it is clear that $\left| \lambda \right| = p$ and, as at most $l$ of the $i_\nu$'s can have the same value for $\smash{\lceil i_\nu/l \rceil}$, the integers $\lambda_i$ are bounded above by $l$. Therefore, the coefficient $a_\lambda$ must be zero unless $\lambda$ is a partition of $p$ which fits into a $p \times l$-box. For such a partition $\lambda$ of $p$, there are exactly $\smash{\prod_i \binom{l}{\lambda_i}}$ ways to write the monomial $x^\lambda$ as a product $x_{\lceil i_1/l \rceil}\ldots x_{\lceil i_p/l \rceil}$: the indexes $i_1,\ldots,i_{\lambda_1}$ must be contained in $\{1,\ldots,l\}$, the numbers $i_{\lambda_1+1},\ldots,i_{\lambda_1+\lambda_2}$ in $\{l+1,\ldots,2l\}$, and so on. This proves that $a_\lambda$ is the desired coefficient.
\end{proof}

The function $f = \sum_\lambda \big( \prod_{i \geq 1} \binom{l}{\lambda_i} \big) m_\lambda$ can be displayed as an integral linear combination $\sum_\mu b_\mu e_\mu$ of elementary symmetric functions; again $\mu$ ranges over partitions of $p$. We can determine these coefficients explicitly (and read off that they are non-negative).

\begin{lem} \label{lem_Michael}
	The following equality holds in the ring of symmetric functions:
	$$
		\sum_\lambda \left( \prod_{i \geq 1} \binom{l}{\lambda_i} \right) m_\lambda = \sum_\mu \left( \prod_{j \geq 1} \binom{l-\mu'_1+\mu'_j}{\mu'_j-\mu'_{j+1}} \right) e_\mu.
	$$
	The first summation ranges over all partitions $\lambda$ of $p$ with $l \geq \lambda_1$ while the second runs over all partitions $\mu$ of $p$ whose length is at most $l$.
\end{lem}

The proof of this lemma was done and explained to the first author by Michael Ehrig. If the following presentation is unclear then this is due to the first author's lack of knowledge of the categorification methods therein.

\begin{proof}
	The proof uses skew Howe duality (cf.\ \cite{how}). Consider the vector space
	$$
		\bigwedge^p(\Cn^r \otimes \Cn^l)
	$$
	as a $\mathfrak{gl}_r$- and as a $\mathfrak{gl}_l$-module. These actions commute. As a $\mathfrak{gl}_r$-module, it decomposes as
	$$
		\bigoplus_{\alpha_1 + \ldots + \alpha_l = p} \bigwedge^{\alpha_1} \Cn^r \otimes \ldots \otimes \bigwedge^{\alpha_l} \Cn^r.
	$$
	In this decomposition, $\alpha_i = 0$ is allowed.
	The character of the module $\bigwedge^{\alpha_1} \Cn^r \otimes \ldots \otimes \bigwedge^{\alpha_l} \Cn^r$ is $e_{\alpha_1}\ldots e_{\alpha_l}$ evaluated at $x_1,\ldots,x_r$ (see \cite[Lecture 24]{fh} for an introduction to characters). Reordering $e_{\alpha_1}\ldots e_{\alpha_l}$ as $e_\mu = e_{\mu_1} \ldots e_{\mu_l}$ for a partition $\mu$, we see that the character of $\bigwedge^p(\Cn^r \otimes \Cn^l)$ as a $\mathfrak{gl}_r$-module is $\sum_{\mu \vdash p} b_\mu e_\mu$ for some non-negative integers $b_\mu$. We can compute these numbers explicitly: $b_\mu$ is the number of tuples $\alpha = (\alpha_1,\ldots,\alpha_l)$ summing to $p$ which can be reordered to $\mu$ (from which it is evident that $b_\mu$ is non-zero if and only if $\mu$ is a partition of $p$ of length at most $l$). Displaying $\mu$ as $1^{n_1}2^{n_2}\ldots$ (i.e. $n_j = \mathrm{mult}_j(\mu)$), we may set $n_1$ of the $\alpha_i$'s to be $1$, then $n_2$ of the remaining $\alpha_i$'s to be $2$, and so forth. In total we get
	$$
		\prod_{j \geq 1} \binom{l - (n_1 + \ldots + n_{j-1})}{n_j} = \prod_{j \geq 1} \binom{l - \mu'_1 + \mu'_j}{\mu'_j-\mu'_{j+1}}
	$$
	possible ways to reorder $\mu$. In the above equation, $\mu'$ denotes the conjugate partition. On the other hand, we compute the character of $\bigwedge^p(\Cn^r \otimes \Cn^l)$ by decomposing it into $\mathfrak{gl}_r$-weight spaces. We have
	$$
		\ch\left( \bigwedge^p (\Cn^r \otimes \Cn^l) \right) = \sum_\lambda \dim \left( \bigwedge^p (\Cn^r \otimes \Cn^l) \right)_{\!\!\lambda} \cdot \sum_{\alpha \in W \cdot \lambda} x^{\alpha}
	$$
	where $\lambda = (\lambda_1,\ldots,\lambda_r)$ is a dominant weight (i.e.\ a partition) and $W\cdot \lambda$ is the Weyl group orbit of $\lambda$. 
	Therefore, the sum
	$$
		\sum_{\alpha \in W\cdot\lambda} x^{\alpha}
	$$
	corresponds to the monomial symmetric function $m_\lambda(x_1,\ldots,x_r)$. The $\mathfrak{gl}_r$-weight space $(\bigwedge^p (\Cn^r \otimes \Cn^l))_\lambda$ is the $\mathfrak{gl}_l$-module
	$
		\bigwedge^{\lambda_1} \Cn^l \otimes \ldots \otimes \bigwedge^{\lambda_r} \Cn^l
	$
	whose dimension is
	$$
		\dim \Big( \bigwedge^{\lambda_1} \Cn^l \otimes \ldots \otimes \bigwedge^{\lambda_r} \Cn^l \Big) = \binom{l}{\lambda_1} \ldots \binom{l}{\lambda_r},
	$$
	which is just the multiplicity of $m_\lambda(x_1,\ldots,x_r)$ in $f(x_1,\ldots,x_r)$.
\end{proof}

The next step is to take the $(s-t)$\textsuperscript{th} power of this expression. We abbreviate $k = s-t$. The product $e_\lambda e_\mu$ of two elementary symmetric functions is $e_{\lambda \cup \mu}$, where $\lambda \cup \mu$ is the partition $1^{m_1 + n_1} 2^{m_2 + n_2} \ldots$ when displaying $\lambda = 1^{m_1}2^{m_2}\ldots$ and $\mu = 1^{n_1}2^{n_2}\ldots$, so the $k$-th power of $f = \smash{\sum_{i=1}^N} b_{\mu^i} e_{\mu^i}$ reads as
$$
	\sum_{k_1 + \ldots + k_N = k} \left( \binom{k}{k_1\ k_2\ \ldots\ k_N} b_{\mu^1}^{k_1} \ldots b_{\mu^N}^{k_N} \right) e_{k_1 * \mu^1 \cup \ldots \cup k_N * \mu^N}.
$$
In the above expression, $k * \mu$ stands for the $k$-fold union $\mu \cup \ldots \cup \mu$. It can be rewritten as
$$
	\sum_\nu \left( \sum_{k_1 + \ldots + k_N = k} \delta_{(k_1 * \mu^1 \cup \ldots \cup k_n * \mu^N, \nu)} \left( \binom{k}{k_1\ k_2\ \ldots\ k_N} b_{\mu^1}^{k_1} \ldots b_{\mu^N}^{k_N} \right) \right) e_\nu = \sum_\nu c_\nu e_\nu.
$$
As the summation $f = \sum_\mu b_\mu e_\mu$ ranges over partitions of $p$ of length bounded by $l$, the sum $\smash{f^k} = \sum_\nu c_\nu e_\nu$ ranges over partitions $\nu$ of $kp$ of length at most $kl$. Moreover, the coefficient of every partition of the form $k*\mu$ is non-zero. All coefficients are obviously non-negative integers.

In order to being able to use the Littlewood--Richardson rule, we need to express $f^k$ in terms of Schur functions $s_\lambda$. The transformation matrix from elementary symmetric functions $e_\nu$ to Schur functions $s_\lambda$ is given by the Kostka numbers, see Lemma \ref{kostka}.
We finally arrive at
$$
	f^k = \sum_\nu c_\nu \sum_\lambda K_{\lambda,\nu}s_{\lambda'} = \sum_\lambda \sum_\nu K_{\lambda',\nu}c_\nu s_{\lambda} = \sum_\lambda d_\lambda s_{\lambda},
$$
the sum ranging over partitions of $kp$. The $d_\lambda$'s are non-negative integers and $d_\lambda$ is positive for example if there exists a partition $\mu$ of $p$ for which $\lambda' \geq k*\mu$.

Let's take a step back and look at what we have done. We have shown that $\pi_{1,*} [Z(\Psi)]$ can be expressed as a $\Zn$-linear combination
$$
	\sum_\lambda d_\lambda s_{\lambda}(x_1,\ldots,x_r) = \sum_\lambda d_\lambda \Delta_\lambda
$$
with non-negative coefficients. Note that, as $Y$ is reduced and irreducible and as $Z(\Psi)$ and $Y$ agree as closed subsets, we get by definition $[Z(\Psi)] = N[Y]$, where $N$ is the length of the local ring of $Z(\Psi)$ at its generic point. As $\pi_1$ induces a birational morphism $Y \to X_1$ (cf.\ Proposition \ref{prop_birat}), we see that $\pi_{1,*}[Y] = [X_1]$, which yields $\pi_{1,*} [Z(\Psi)] = N[X_1]$. The basis theorem (Proposition \ref{basis_thm}) then implies that $N$ is a common divisor of the $d_\lambda$'s.

In order to finally show that $X_1$ and $X_2$ intersect, we have to compute the intersection product $[X_1] \cdot [X_2]$. The subvariety $X_2 = \Gr_r(W)$ is a Schubert variety. Its class is $$\Delta_{(ls-w)^r} = c_r(\UU^\vee)^{ls-w};$$ it is not hard to see that every determinantal locus $\Omega(U_*)$ corresponding to $\Delta_{(ls-w)^r}$ (see subsection \ref{intersection_theory}) is reduced. The multiplication of two Schubert cycles is given by the Littlewood--Richardson rule. In general, this is pretty messy but here, we are in a very favorable situation: we are forced to stay in the $r \times (ls-r)$-box and the partition $(ls-w)^r$ has maximal length. We make use of the following lemma:

\begin{lem} \label{lem_Pieri}
	Let $d \leq n$ and let $\lambda$ be a partition of length no greater than $d$ and with $\lambda_1 \leq n-d$. Let $j \leq n-d$. In the Chow ring $A^*(\Gr_d(k^n))$, we have
	$$
		\Delta_\lambda \cdot \Delta_{j^d} = \Delta_{\lambda + j^d}.
	$$
	The class $\Delta_{\lambda + j^d}$ is non-zero if and only if $\lambda_1+j \leq n-d$.
\end{lem}

In the lemma, the sum $\lambda + \mu$ of two partitions $\lambda$ and $\mu$ is taken component-wise, i.e.\ $\lambda + \mu = (\lambda_1+\mu_1,\lambda_2+\mu_2,\ldots)$.

\begin{proof}
	As $\Delta_{j^d} = c_d(\UU^\vee)^j = \Delta_{1^d}^j$, we are down to showing the assertion for $j=1$. The product $\Delta_\lambda \cdot \Delta_{1^d}$ computes by Pieri's rule (Lemma \ref{pieri}) as
	$$
		\Delta_\lambda \cdot \Delta_{1^d} = \sum_\mu \Delta_\mu,
	$$
	where $\mu$ runs through all partitions arising from $\lambda$ by adding a total number of $d$ boxes, at most one per row. There is only one partition $\mu$ obtained in such a way whose length does not exceed $d$: the one we get by inserting exactly one box in each row. That's precisely the partition $\lambda + 1^d$.
\end{proof}

Applying this lemma to our situation, we obtain that $(\pi_{1,*}[Z(\Psi)]) \cdot [X_2] = \sum_{\lambda} d_\lambda \Delta_{\lambda + {(ls-w)^r}}$, the sum ranging over partitions $\lambda$ of $kp = (lr-t)(s-t)$ contained in a box of size $r \times (ls-r)$, whose width is less than or equal to $ls-r-(ls-w) = w-r$, that means $\lambda$ must actually be contained in an $r \times (w-r)$-box. We have proved:

\begin{thm} \label{thm_inter_prod}
	The intersection product of the subvarieties $X_1$ of all $U \in \Gr_r(V)$ for which a $U_0 \in \Gr_t(V_0)$ with $U \sub U_0^l$ exists and $X_2 = \Gr_r(W)$ in the Chow ring $A^*(\Gr_r(V))$ is
	$$
		[X_1] \cdot [X_2] = \frac{1}{N}\sum_\lambda d_{\lambda} \Delta_{\lambda + (ls-w)^r},
	$$
	where the sum is taken over all partitions $\lambda$ of $(lr-t)(s-t)$ of length at most $r$ and with $\lambda_1 \leq w-r$ and $N$ is the length of the local ring of $Z(\Psi)$ at its generic point. The coefficient $d_\lambda$ is a multiple of $N$ and computes as
	$$
		d_\lambda = \sum_\nu K_{\lambda',\nu} \sum_{\sum_\mu k_\mu = s-t} \delta_{(\bigcup_\mu k_\mu * \mu, \nu)} \left( (s-t)! \prod_\mu \frac{1}{k_\mu!} \left( \prod_{j \geq 1} \binom{l-\mu_1'+\mu_j'}{\mu_j'-\mu'_{j+1}} \right)^{k_\mu} \right).
	$$
	In this expression, $\nu$ is a partition of $(lr-t)(s-t)$ of length at most $l(s-t)$ and $\mu$ runs through all partitions of $lr-t$ of length $l(\mu) \leq l$.
\end{thm}


The coefficients look horrible, that's true, but from the properties of Kostka numbers, we can see that $[X_1] \cdot [X_2] \neq 0$: the coefficients $c_\nu$ are non-negative, and $c_{(s-t)*\mu}$ is positive for every $\mu \vdash lr-t$ with $l(\mu) \leq l$. Therefore, it suffices to show that there exists a partition $\lambda$ inside the box $r \times (w-r)$ whose conjugate $\lambda'$ dominates a partition of the form $(s-t)*\mu$. Take $\mu$ to be minimal among these partitions. That means
$$
	\mu = q^{l-j}(q+1)^j,
$$
where we choose $q \in \Zn_{\geq 0}$ and $j \in \{0,\ldots,l-1\}$ such that $lq+j = lr-t$. Then, as $t > 0$, we get $q < r$. We obtain $(s-t)*\mu = q^{(l-j)(s-t)}(q+1)^{j(s-t)}$. Moreover, let $\lambda'$ be the maximal possible partition (inside the box of size $(w-r) \times r$, as we are dealing with the conjugate of $\lambda$), that is
$$
	\lambda' = k^1r^m,
$$
where $m \in \Zn_{\geq 0}$ and $k \in \{0,\ldots,r-1\}$ such that $rm+k = (lr-t)(s-t)$. Then $\lambda$ and $(s-t)*\mu$ are partitions of the same number and the fact that $r \geq q+1$ (see above) yields $\lambda \geq (s-t)*\mu$. This shows that $[X_1] \cdot [X_2] \neq 0$, so we finally arrive at

\begin{kor} \label{cor_non_empty}
	The subvarieties $X_1$ and $X_2$ of $\Gr_r(V)$ intersect.
\end{kor}

\subsection{Indecomposability}\label{glueing}
\noindent In this section, we treat the third part of Question \ref{mainque}. To do so, we keep the notation and conventions from the last sections and fix a non-Schurian root $\alpha$ of a quiver $Q(\un{m})$ of type one with canonical exceptional decomposition $\alpha=\alpha_1^{d_1}+\alpha_2^{d_2}+\alpha_3^{d_3}$ such that $\hat\alpha=\alpha_2^{d_2}+\alpha_3^{d_3}$ is an imaginary root. We have seen in Remark \ref{rem1} that, if $d_3=1$, every representation in $I^{\alpha}$ is indecomposable, even Schurian. If $d_3\geq 2$, it is not clear at all which points in $I^{\alpha}$ correspond to indecomposable representations. The aim of this section is to apply the methods of section \ref{Krone} in order to show that there exists an open subset of Schurian representations in $I^\alpha$ for a large number of roots, even if $d_3\geq 2$. This assures that those representations corresponding to points of $I^\alpha$ can indeed be used to construct indecomposable representation of dimension $\alpha$, giving a (partial) answer to the third part of Question \ref{mainque}. The following lemma is crucial for the main result of this section: 
\begin{lem}\label{genhom}Let $\alpha=\beta+\gamma$ be a decomposition into non-Schurian roots of type one which is compatible with the canonical decomposition of $\alpha$, i.e.
$\beta=\alpha_1^{c_s}\oplus\hat\beta$, $\gamma=\alpha_1^{c}\oplus\hat\gamma$ for certain $c_s,c\geq 0$ and imaginary roots $\hat\beta,\hat\gamma\in\alpha_1^\perp$.
Then the map $\dim\Hom:I_\beta\times I_\gamma\to\Nn,\,(S,T)\mapsto\dim\Hom(S,T)$ is semi-continuous. 
\end{lem}
\begin{proof}
We denote the two indecomposable projective representations of $K(n)$ by $P_0$ and $P_1$ respectively. We have $\udim P_0=(1,n)$ and $\udim P_1=(0,1)$. Every representation $M\in R_{(d,e)}(K(n))$ such that $M$ has no direct summand isomorphic to $P_1$ -- this is equivalent to $\Hom(M,P_1)=0$ -- has a minimal projective resolution of the form
\[\ses{P_1^{ne-d}}{P_0^d}{M}.\] 
This defines an open subset $R^o_{(d,e)}(K(n))$ of $R_{(d,e)}(K(n))$ on which we have a natural $\mathrm{Gl}_{ne-d}(k)$-action whose quotient is the Grassmannian $\Gr_{ne-d}(\Hom(P_1,\smash{P_0^d}))$.
Analogous to \cite[section 1]{sch}, we can consider the quasi-projective (even quasi-affine) variety
\[\{(f,M,N)\mid f\in\Hom(M,N)\}\subset (k^{d_sd}\times k^{e_se})\times R_{(d_s,e_s)}^o(K(n))\times R_{(d,e)}^o(K(n))\]
together with the projection onto $R_{(d_s,e_s)}^o(K(n))\times R_{(d,e)}^o(K(n))$. Then the fibre of $(M,N)$ is $\Hom(M,N)$. Thus $\dim\Hom$ is semi-continuous on $R_{(d_s,e_s)}^o(K(n))\times R_{(d,e)}^o(K(n))$. 

We have $I^\alpha\subset\Gr_r(W)$ where $W=\Hom(S_{\alpha_2},S_{\delta})$ and $r=nd_3-d_2$. As already mentioned in Remark \ref{proj} the representations $S_{\alpha_2}$ and $S_\delta$ are the indecomposable projective representations in $S_{\alpha_1}^\perp$. The equivalence of categories between $S_{\alpha_1}^\perp$ and $\Rep(K(n))$ gives rise to an isomorphism between the Grassmannian $\Gr_r(W)$ and $R_{(d_2,d_3)}^o(K(n))/\mathrm{Gl}_{r}(k)$. If $\alpha=\beta+\gamma$ is a decomposition which is compatible with the canonical decomposition of $\alpha$, together with the previous considerations this shows that the semi-continuity of $\dim\Hom$ is preserved when passing to the varieties $I^\beta$ and $I^\gamma$.
\end{proof}

Inspired by Definition \ref{homorth} we introduce the following notion:
\begin{defi} \label{homorth2}
\begin{enumerate} \item We call a decomposition $\alpha=\beta+\gamma$ of $\alpha$ into roots $\Hom$-orthogonal if 
\begin{enumerate}[label={\roman*)}]
\item $\beta=\alpha_1^{c_s}+\alpha_2^{d_s+kd}+\alpha_3^{e_s+ke}$ and $\gamma=\alpha_1^{c}+\alpha_2^{d}+\alpha_{3}^e$ are the canonical exceptional decompositions for certain $c_s,d_s,e_s,c,d,e,k\geq 0$;
\item $\beta$ is a root of type one of $Q(\un m)$ or $c_s=0$ and $\gamma$ is a root of type one of $Q(\un m)$ or $c=0$ ; 
\item$((d_s+kd,e_s+ke),(d,e))$ is a $\Hom$-orthogonal pair of $K(\ext(\alpha_3,\alpha_2))$. 
\end{enumerate}
\item Let $M$ be an exceptional representation and $N,T\in M^{\perp}$ (resp. $^\perp M$). A short exact sequence $e\in\Ext(T,N)$ with middle term $S$ is called $M$-additive if 
\begin{eqnarray*}\dim\Hom(S,M)&=&\dim\Hom(N,M)+\dim\Hom(T,M),\\
\dim\Ext(S,M)&=&\dim\Ext(N,M)+\dim\Ext(T,M).
\end{eqnarray*}
If, additionally, the middle term of $e$ is indecomposable, we call $e$ indecomposable.
\end{enumerate}

\end{defi}
If $c_s=0$ (resp. $c=0$), we have that $\hat\beta=\beta\in\alpha_1^\perp$ (resp. $\hat\gamma=\gamma\in\alpha_1^\perp$). Then every representation of dimension $\hat\beta$ (resp. $\hat\gamma$) satisfies equation (\ref{eq1}) of section \ref{haupt}. Using the language of this definition and section \ref{Krone} this means that, in order to construct an indecomposable $M_{\hat\alpha}$ in $I^\alpha$, it suffices to find a $\Hom$-orthogonal decomposition $\alpha=\beta+\gamma$ and two general Schurian representations $M_{\hat\beta}\in I^\beta$ and $M_{\hat\gamma}\in I^\gamma$ which can be glued along indecomposable $S_{\alpha_1}$-additive exact sequences using Theorem \ref{thm1}. We make this precise in the following. By general representation we mean that the corresponding $\Hom$-spaces vanish as predicted by the decomposition. 

We first stick to the problem of indecomposable additive  exact sequences. Thus let be $M$ an exceptional representation and, moreover, let $N\in M{^\perp}$ be indecomposable, but not exceptional. By Lemma \ref{insu}, there exists a short exact sequence 
\[0\to L\to N \xrightarrow{\phi} M^n\to 0\]
induced by a basis $(\phi_1,\ldots,\phi_n)$ of $\Hom(N,M)$. For an arbitrary representation $T$, applying $\Hom(T,\blank)$, we consider the following part of the respective long exact sequence
\[\Ext(T,L)\to\Ext(T,N)\xrightarrow{f_{T,N}}\Ext(T,M^n)\to 0.\]
If $f_{T,N}(e)=0$, we get a commutative diagram
\[
\begin{xy}\xymatrix@R20pt@C20pt{e:&0\ar[r]&N\ar[r]\ar@{->>}[d]^\phi&S\ar[r]\ar[d]&T\ar[r]&0\\f_{T,N}(e):&0\ar[r]&M^n\ar[r]&T\oplus M^n\ar[r]&T\ar@{=}[u]\ar[r]&0\\}
\end{xy}
\]
for some representation $S$. Since $(\phi_1,\ldots,\phi_n)$ is a basis of $\Hom(N,M)$, applying the functor $\Hom(\underline{\,\,\,\,},M)$ to the first row, this means that for the induced map $g:\Hom(N,M)\to\Ext(T,M)$ we have $g=0$. In particular, the sequence $e$ is $M$-additive if $f_{T,N}(e)=0$.

We transfer these considerations to our situation. Therefore, we assume that $\alpha=\beta+\gamma$ is a $\Hom$-orthogonal decomposition with $\hat\beta=\alpha_2^{d_s+kd}+\alpha_3^{e_s+ke}$ and $\hat\gamma=\alpha_2^d+\alpha_3^e$. Define
\[n_{\beta}=\Sc{\hat\beta}{\alpha_1},\quad n_\gamma=\Sc{\hat\gamma}{\alpha_1}.\]
Then we have that for a general representation $M_{\hat\beta}\in I^{\beta}$ (resp. $M_{\hat\gamma}\in I^\gamma$) there exist short exact sequences
\[\ses{N_{\hat\beta}}{M_{\hat\beta}}{S_{\alpha_1}^{n_\beta+c_s}}\quad \text{(resp. }\ses{N_{\hat\gamma}}{M_{\hat\gamma}}{S_{\alpha_1}^{n_\gamma+c}})\]
yielding
\[0\to\Hom(M_{\hat\beta},S_{\alpha_1}^{n_\gamma+c})\to\Ext(M_{\hat\beta},N_{\hat\gamma})\xrightarrow{g_{\hat\beta,\hat\gamma}}\Ext(M_{\hat\beta},M_{\hat\gamma})\xrightarrow{f_{\hat\beta,\hat\gamma}}\Ext(M_{\hat\beta},S_{\alpha_1}^{n_\gamma+c})\to 0.\]
From now on assume that there exists a pair of Schurian representations $M=(M_{\hat\beta}, M_{\hat\gamma})$ with $M_{\hat\beta}\in I^{\beta}$ and $M_{\hat\gamma}\in I^{\gamma}$ satisfying $\Hom(M_{\hat\beta},M_{\hat\gamma})=\Hom(M_{\hat\gamma},M_{\hat\beta})=0$. Note that this condition is not part of Definition \ref{homorth2}; but the last condition of the definition is necessary for finding such a pair.

By Theorem \ref{thm1}, every such pair of Schurian representations gives rise to a fully faithful embedding of $\Rep(Q(M))$ into $\Rep(Q(\un m))$. Here $Q(M)$ has vertices $Q(M)_0=\{m_{\hat\beta},m_{\hat\gamma}\}$ and $-\Sc{\hat\beta}{\hat\gamma}$ arrows from $m_{\hat\beta}$ to $m_{\hat\gamma}$ and $-\Sc{\hat\gamma}{\hat\beta}$ arrows in the other direction. The arrows are in one-to-one correspondence to a basis $\{\chi^1_1,\ldots,\chi^1_{\ext(\hat\beta,\hat\gamma)}\}$ of $\Ext(M_{\hat \beta},M_{\hat\gamma})$ (resp. $\{\chi^2_1,\ldots,\chi^2_{\ext(\hat\gamma,\hat\beta)}\}$ of $\Ext(M_{\hat\gamma},M_{\hat\beta})$). We can choose these bases in such a way that 
\[\{\chi^1_1,\ldots,\chi^1_{\dim\ker(f_{\hat\beta,\hat\gamma})}\}\text{ (resp. }\{\chi^2_1,\ldots,\chi^2_{\dim\ker(f_{\hat\gamma,\hat\beta})}\})\]
is a basis of $\ker(f_{\hat\beta,\hat\gamma})$ (resp. $\ker(f_{\hat\gamma,\hat\beta})$). Now the exact sequences $e\in \ker(f_{\hat\gamma,\hat\beta})$ (resp. $e\in \ker(f_{\hat\beta,\hat\gamma})$) are $S_{\alpha_1}$-additive. 
Thus together with Lemma \ref{sinksource} we can conclude:

\begin{pro}\label{propo}Every indecomposable representation $X\in\Rep(Q(M))$ of dimension $(1,1)$ such that $X_{\chi_j^i}=0$ if $i=1$ and $j>\dim\ker(f_{\hat\beta,\hat\gamma})$ or $i=2$ and $j>\dim\ker(f_{\hat\gamma,\hat\beta})$ defines a Schurian representation $F_M(X)$ of $Q(\un m)$ such that $\Hom(F_M(X),S_{\alpha_1})\geq\Sc{\alpha_2^{d_2}+\alpha_3^{d_3}}{\alpha_1}+d_1$. 
\end{pro}
If $Q(M)^0$ denotes the subquiver of $Q(M)$ having the same vertices and only $\dim\ker(f_{\hat\beta,\hat\gamma})$ (resp. $\dim\ker(f_{\hat\gamma,\hat\beta})$) arrows in the respective directions, the representations considered in Proposition \ref{propo} obviously correspond to representations of $Q(M)^0$. For every such pair $M$, we thus get a $(1-\Sc{(1,1)}{(1,1)}_{Q(M)^0})$-parameter family of isomorphism classes of representations of $I^\alpha$. As $M_{\hat\beta}$ and $M_{\hat\gamma}$ are Schurian, up to the representation corresponding to the semi-simple representation of $Q(M)$ of dimension $(1,1)$, the glued representations are also Schurian and contained in $I^\alpha$. Since $\Sc{\alpha_1}{\hat\gamma}=0$, we have
\begin{eqnarray*}\dim\ker(f_{\hat\beta,\hat\gamma})&=&\dim\Ext(M_{\hat\beta},M_{\hat\gamma})-\dim\Ext(M_{\hat\beta},S_{\alpha_1}^{n_\gamma+c})\\&=&-\Sc{\beta-\alpha_1^{c_s}}{\gamma-\alpha_1^{c}}-(c+n_\gamma)c_s\\&=&-\Sc{\beta}{\gamma}+\Sc{\beta}{\alpha_1^{c}}-(c+n_\gamma)c_s\\&=&-\Sc{\beta}{\gamma}+\Sc{\hat\beta}{\alpha_1^{c}}-\Sc{\hat\gamma}{\alpha_1^{c_s}}.\end{eqnarray*}
Analogously, we obtain 
\[\dim\ker(f_{\hat\gamma,\hat\beta})=-\Sc{\gamma}{\beta}+\Sc{\hat\gamma}{\alpha_1^{c_s}}-\Sc{\hat\beta}{\alpha_1^{c}}.\]
Thus we have
\[\dim\ker(f_{\hat\beta,\hat\gamma})+\dim\ker(f_{\hat\gamma,\hat\beta})=-\Sc{\beta}{\gamma}-\Sc{\gamma}{\beta}.\]
Taking into account Remark \ref{kacgen}, the irreducible components of $I^\beta$ (resp. $I^\gamma$) containing $M_{\hat \beta}$ (resp. $M_{\hat\gamma}$) already contain a $(1-\Sc{\beta}{\beta})$-parameter family (resp. a $(1-\Sc{\gamma}{\gamma})$-parameter family) of isomorphism classes of Schurian representations. Thus we can construct a parametric family of isomorphism classes of Schurian representations of $I^\alpha$  with the following number of parameters when applying the functors $F_M$ (see also \cite[section 3.1.2]{wei4}):
\begin{eqnarray*}&&1-\Sc{\beta}{\beta}+1-\Sc{\gamma}{\gamma}+1-\Sc{(1,1)}{(1,1)}_{Q(M)^0}\\
&=&1-\Sc{\beta}{\beta}+1-\Sc{\gamma}{\gamma}+1-(1+1-(\dim\ker(f_{\hat\beta,\hat\gamma})+\dim\ker(f_{\hat\gamma,\hat\beta})))\\&=&1-\Sc{\beta}{\beta}+1-\Sc{\gamma}{\gamma}+(-1-\Sc{\beta}{\gamma}-\Sc{\gamma}{\beta})\\&=&1-\Sc{\alpha}{\alpha}\end{eqnarray*}
which is number predicted by Kac's Theorem. Summarizing, we obtain:
\begin{thm}\label{glueing1}
Let $\alpha$ be a non-Schurian root of $Q(\un m)$ with canonical exceptional decomposition $\alpha=\alpha_1^{d_1}+\alpha_2^{d_2}+\alpha_3^{d_3}$ such that $(d_2,d_3)$ is coprime. Moreover, let $\alpha=\beta+\gamma$ be a $\Hom$-orthogonal decomposition such that there exist two Schurian representations $M_{\hat\beta}\in I^\beta$ and $M_{\hat\gamma}\in I^\gamma$ with $\Hom(M_{\hat\beta},M_{\hat\gamma})=\Hom(M_{\hat\gamma},M_{\hat\beta})=0$. 
Then there exists an irreducible component of $I^\alpha$ of dimension $d_3^2-\Sc{\alpha}{\alpha}$ containing an open subset of Schurian representations which can be obtained recursively from Schurian representations of dimensions $\hat{\beta}$ and $\hat{\gamma}$ when applying Theorem \ref{thm1}.
\end{thm}

For the non-coprime case we need a slight generalization. Assume that we have $(d_2,d_3)=n(\hat d_2,\hat d_3)$ with $n=\mathrm{gcd}(d_2,d_3)$. Then we can decompose $(\hat d_2,\hat d_3)$ as done before. Moreover, a general representation of dimension $(\hat d_2,\hat d_3)$ is Schurian. Since $(\hat d_2,\hat d_3)$ is also imaginary, by \cite[Theorem 3.5]{sch}, we have $\hom((\hat d_2,\hat d_3),(\hat d_2,\hat d_3))=0$. From this we obtain $\hom(k(\hat d_2,\hat d_3),l(\hat d_2,\hat d_3))=0$ for $k,l\geq 1$. Thus we obtain:

\begin{thm}\label{glueing2}
Let $\alpha$ be a non-Schurian root of $Q(\un m)$ with canonical exceptional decomposition $\alpha=\alpha_1^{d_1}+\alpha_2^{d_2}+\alpha_3^{d_3}$ such that $(d_2,d_3)=n(\hat d_2,\hat d_3)$ where $(\hat d_2,\hat d_3)$ is coprime. Assume that there exist decompositions $d_1=c+c'$ and $n=k+l$ with $k,l\geq 1$ such that 
\begin{enumerate}
\item $\beta=\alpha_1^{c}+\alpha_2^{k\hat d_2}+\alpha_3^{k\hat d_3}$ and $\gamma=\alpha_1^{c'}+\alpha_2^{l\hat d_2}+\alpha_{3}^{l\hat d_3}$ are roots of type one of $Q(\un m)$; 
\item there exist two Schurian representations $M_{\hat\beta}\in I^\beta$ and $M_{\hat\gamma}\in I^\gamma$ such that we have $\Hom(M_{\hat\beta},M_{\hat\gamma})=\Hom(M_{\hat\gamma},M_{\hat\beta})=0$.

\end{enumerate}
Then there exists an irreducible component of $I^\alpha$ of dimension $d_3^2-\Sc{\alpha}{\alpha}$ containing an open subset of Schurian representations which can be obtained recursively from Schurian representations of dimensions $\hat{\beta}$ and $\hat{\gamma}$ when applying Theorem \ref{thm1}.
\end{thm}

Clearly, the main open question which remains is whether there exists an indecomposable or even a Schurian representation in $I^\alpha$ for every non-Schurian root $\alpha$ of type one. We think that this is always true:
\begin{conj}
If $\alpha$ is a non-Schurian root of type one, the variety $I^{\alpha}$ contains a Schurian representation.
\end{conj}

It would also be interesting to study what happens when taking stability into account. Since the set of stable representations is open in the affine space of representations of a fixed dimension, most of the arguments transfer when replacing Schurian by stable. In particular, if there exists one stable representation in $I^{\alpha}$, there already exists an open subset.
But this is actually the critical point: how can we construct one stable representation starting by glueing (stable) representations $M_{\hat\beta}\in I^{\beta}$ and $M_{\hat\gamma}\in I^{\gamma}$. It is likely, but needs to be checked in detail, that the methods of \cite[section 4.3]{wei} can be applied  to construct such representations.
\section{Special cases and examples}\label{examples}
\subsection{Special cases}\label{special}
In this section, we frequently use the notation of section \ref{three}. There we restricted to non-Schurian roots $\alpha=\alpha_1^{d_1}+\alpha_2^{d_2}+\alpha_3^{d_3}$ of type one, i.e. we have 
\[r\leq nd_3\leq ls\text{ and } r\leq lt.\]
Here we have $\dim V=ls$ and $\dim W=nd_3$. The condition $r\leq lt$ assures that the flag variety $\Fl_{(r,lt)}(V)$ is well-defined. In the following, we want to study which consequences it has if at least one of these conditions is not satisfied. In most of the cases, the conditions give rise to inequalities which seem to be satisfied only for a small number of roots. Actually, in some cases we conjecture that there exists no root which satisfies the given inequalities.

\subsubsection{Condition $r\leq nd_3-l(d_1+\Sc{\hat\alpha}{\alpha_1})$}\label{Fall0}
This inequality seems to be satisfied only for a small number of roots. If it is satisfied, we can consider short exact sequences of the form
\begin{eqnarray}\label{ses1}0\to N\rightarrow S_\delta^{d_3}\xrightarrow{\pi} S_{\alpha_1}^{d_1+\Sc{\hat\alpha}{\alpha_1}}\to 0\end{eqnarray}
with $N:=\ker(\pi)$. This yields a long exact sequence
\[0\to\Hom(S_{\alpha_2},N)\to\Hom(S_{\alpha_2},S_\delta^{d_3})\to\Hom(S_{\alpha_2},S_{\alpha_1}^{d_1+\Sc{\hat\alpha}{\alpha_1}})\to\Ext(S_{\alpha_2},N)\to 0\]
with $r\leq nd_3-l(d_1+\Sc{\hat\alpha}{\alpha_1})\leq \dim \Hom(S_{\alpha_2},N)\leq nd_3.$

We define $W:=\Hom(S_{\alpha_2},S_\delta^{d_3})$ and, moreover, $Z:=\Hom(S_{\alpha_2},N)$. 
Keeping in mind the universal property of the cokernel, every point in the subvariety 
$$X^{\alpha}(N):=\{U\in\Gr_r(W)\mid U\subset Z\}$$ induces a commutative diagram
\[
\begin{xy}\xymatrix@R20pt@C20pt{0\ar[r]&S_{\alpha_2}^r\ar[r]\ar[d]&S_\delta^{d_3}\ar[r]\ar[d]^{\pi}&M_{\hat\alpha}\ar[r]\ar[d]&0\\0\ar[r]&0\ar[r]&S_{\alpha_1}^{d_1+\Sc{\hat\alpha}{\alpha_1}}\ar@{=}[r]&S_{\alpha_1}^{d_1+\Sc{\hat\alpha}{\alpha_1}}\ar[r]&0\\}
\end{xy}
\]
where $\udim  M_{\hat\alpha}=\hat\alpha$. In particular, every point induces a surjection $M_{\hat\alpha}\to S_{\alpha_1}^{d_1+\Sc{\hat\alpha}{\alpha_1}}$. In this case, the dimension of the variety $X^{\alpha}(N)$ is easily determined as $$\dim X^{\alpha}(N)=\dim\Gr_r(Z)\geq r(d_2-l(d_1+\Sc{\hat\alpha}{\alpha_1})).$$ We obtain:

\begin{lem}
Every  representation $N$ as constructed above gives rise to a variety $X^{\alpha}(N)$ of dimension at least $r(d_2-l(d_1+\Sc{\hat\alpha}{\alpha_1}))$ such that every point in this variety corresponds to a representation $M_{\hat\alpha}\in S_{\alpha_1}^\perp$ with $\dim\Hom(M_{\hat\alpha},S_{\alpha_1})=d_1+\Sc{\hat\alpha}{\alpha_1}$.
\end{lem}
\begin{rem}                                                     

Note that it is not at all clear under which conditions non-isomorphic representations $N$ and $N'$ yield that arbitrary pairs of representations in $X^{\alpha}(N)$ and $X^{\alpha}(N')$ are non-isomorphic. Thus it seems to be more difficult to say something about the dimension of indecomposables which can potentially be constructed with these methods. 

\end{rem}

\begin{exam}We consider $Q(1,1,1)$ and the real root $(d,d+1,d)$ which has the canonical exceptional decomposition
\[(d,d+1,d)=(0,1,0)\oplus (1,1,0)^d\oplus (0,0,1)^d.\]
In particular, it is not a Schur root for $d\geq 1$. Then we have $l=m=1$ and $n=2$. Thus it follows that
$$r=nd_3-d_2=d\leq nd_3-l(d_1+\Sc{\hat\alpha}{\alpha_1})=2d-1.$$
Note that we also have $nd_3=2d>d=ls$ and $r=d\leq d=lt$ in this case. Note further that this case is again special because $(d,d)$ is an isotropic root of $K(n=2)$ which is not Schurian if $d\geq 2$. If $d=1$, up to equivalence there exists only one short exact sequence of the form (\ref{ses1}). Moreover, we have $X^{\alpha}(N)=\{\mathrm{pt}\}$ in this case and the upper row of the commutative diagram is given by
$$\begin{xy}\xymatrix@R10pt{&&\bullet&&&\bullet&&\bullet&&&&\bullet\\0\ar[r]&&&\ar[r]&&&&&\ar[r]&&&&\ar[r]&0\\&&\bullet\ar[uu]^{\rho_1}&&&\bullet\ar[uu]^{\rho_1}&\bullet\ar[l]^{\rho_3}\ar[ruu]^{\rho_2}&\bullet\ar[uu]_{\rho_1}&&&\bullet\ar[ruu]^{\rho_2}&\bullet\ar[uu]^{\rho_1}}
\end{xy}$$
Clearly, we have $\Hom(M_{\hat\alpha},S_{\alpha_1})=k$. Now, for general $d$ and $\pi$ the kernel of $\pi$ is decomposable and we have $X^{\alpha}(\ker\pi)=\Gr_d(k^{2d-1})$. Since there only exists one indecomposable of dimension $(d,d+1,d)$ up to isomorphism, this already suggests that the methods in this case need to be improved. Note that, in this case we can construct the indecomposable representation $M_{\hat\alpha}$ with $\Hom(M_{\hat\alpha},S_{\alpha_1})=k$ by glueing the coefficient quivers of the representation on the right hand side $(d-1)$-times to itself according to its self-extensions.

\end{exam}

\subsubsection{Conditions $nd_3>ls$ and $r\leq lt$}\label{Fall2}
This condition says that every morphism $f:S_{\alpha_2}^r\to S_{\alpha_1}^s$ factors through $S_{\delta}^{d_3}$.  Indeed, as $\alpha_1, \delta$ are exceptional in $^\perp\alpha_2$ and $f_2:S_\delta^{d_3}\to S_{\alpha_1}^s$ is induced by a basis of $\Hom(S_\delta,S_{\alpha_1})$, it follows that the kernel $K$ of $f_2$ is also a multiple of  an exceptional representations with $K\in{^\perp}\alpha_2$. Thus, Theorem \ref{schofield} together with
\[nd_3=\dim\Hom(S_{\alpha_2},S_\delta^{d_3})\geq \dim\Hom(S_{\alpha_2},S_{\alpha_1}^s)=ls\]
implies the claim as $\Ext(S_{\alpha_2},K)=0$. This would make things easier because we would have $I^{\alpha}=X^{\alpha}_1$ in this case.

But note that, since we have $s=(nl-m)d_3$, the first inequality is equivalent to $n(l^2-1)<lm$. Now we have $nl>m$ because $\Sc{\delta}{\alpha_1}\geq 0$. Thus except for the case $l=1$, this seems to be very restrictive in the sense that only a few number of roots satisfy this condition. But if $l=1$, we already have $t=lr-d_1=r-d_1\geq r$ which is not possible. Thus we have $l\geq 2$. In this case, it would be interesting to study which roots satisfy the given inequalities or if there even exist such roots.

\subsubsection{Conditions $nd_3>ls$ and $lt<r$}\label{Fall1}
Because of the considerations from case \ref{Fall2}, the inequalities suggest to consider the case $l=1$. But in this case, we conjecture that we already have $r\leq nd_3-l(d_1+\Sc{\hat\alpha}{\alpha_1})$ and thus we were faced with case \ref{Fall0}. Indeed, if this were not the case, we had
\[nd_3-d_2>nd_3-d_1-d_2+md_3\Leftrightarrow d_1>md_3.\]
We again conjecture that this cannot be true. Roughly speaking, the canonical exceptional decomposition suggests that all extensions $e\in\Ext(M_{\hat\alpha},S_{\alpha_1})$ are induced by extensions $e'\in\Ext(S_{\alpha_3},S_{\alpha_1})$. But if $d_1>md_3$, there are in a sense too few extensions to have that $\alpha$ is a root.

\subsubsection{Conditions $nd_3\leq ls$ and $lt<r$}\label{Fall3}
Since $nd_3\leq ls$ and $s=(nl-m)d_3$, we have $l\geq 2$. Moreover, since $t=lr-d_1$, we have
\[r>lt\Leftrightarrow ld_1> (l^2-1)r\Leftrightarrow ld_1+(l^2-1)d_2>(l^2-1)nd_3\Leftrightarrow \frac{l}{l^2-1}d_1+d_2>nd_3.\]
Actually, no root satisfying these two inequalities is known to us. For similar reasons as in the previous case, we even conjecture that there exists no root which satisfies this inequality.

\subsection{Examples} \label{exs}
In this subsection we consider several examples which are to illustrate the introduced methods.
\begin{exam} We consider the following concrete example: let $\underline{m} = (2,1,1)$ and consider the root $$(2,5,2)=(0,1,0)+(1,2,0)^2+(0,0,1)^2.$$ 
This leads to the following values:
$$
	n=3,\ \ m=1,\ \ l=2,\ \ r=4,\ \ s=10,\ \ t=7,\ \ \text{and } w=6.
$$ 
Remember that $\dim V_0 = s = 10$, $\dim V = ls = 20$ and $\dim W = w = 6$. Let's compute the intersection product $(\pi_{1,*}[Z(\Psi)]) \cdot [X_2]$ by hand in this case in order to establish some trust in the computations from subsection \ref{intsec}. On the Grassmannian $\Gr_r(V) = \Gr_4(k^{20})$ there is the universal subbundle $\UU$ of rank $4$. On $\Gr_t(V_0) = \Gr_7(k^{10})$, we need the universal quotient bundle $\QQ_0$ of rank $3$. The Chern roots of $\UU^\vee$ are denoted by $x_1,\ldots,x_4$ and $y_1,y_2,y_3$ will be the Chern roots of $\QQ_0$. We have
$$
	\Zn(\Psi) = \prod_{i=1}^4 \prod_{j=1}^3 (y_j+x_i)^2,
$$
when, by abuse of notation, we identify $\Zn(\Psi)$ with its image in $A^*(\Gr_4(k^{16}))$. The coefficient of $c_3(\QQ_0)^7 = (y_1y_2y_3)^7$ is 
$$
	8 m_{1^1}^3 = 8e_1^3 = 8e_{1^3}
$$
(compare this to Lemma \ref{lem_Michael}). We need to express $e_{1^3}$ in terms of Schur functions. We have $e_{1^3} = \sum_\lambda K_{\lambda',1^3}s_\lambda$. The integer $K_{\lambda',1^3}$ is the number of semi-standard Young tableaux of shape $\lambda'$ and weight $1^3$, that's what's usually called a standard Young tableau of shape $\lambda'$. The number $K_{\lambda'}$ of standard Young tableaux of shape $\lambda'$ equals the number $K_\lambda$. The possible tableaux are
$$
	\young(1,2,3),\quad \young(12,3),\ \young(13,2),\ \ \text{ and } \young(123)\,.
$$
This leads to
$$
	\pi_{1,*}\Zn(\Psi) = 8\Delta_{1^3} + 16\Delta_{1^12^1} + 8\Delta_{3^1}.
$$
We know that $[X_2] = \Delta_{14^4}$ and the multiplication is given by addition of partitions by Lemma \ref{lem_Pieri}. We have
\begin{align*}
	1^3 + 14^4 &= 14^1 15^3, \\
	1^1 2^1 + 14^4 &= 14^2 15^1 16^1, \text{ and} \\
	3^1 + 14^4 &= 14^3 17^1,
\end{align*}
the last of which is not contained in the box of size $4 \times 16$. Therefore, we get
$$
	(\pi_{1,*}\Zn(\Psi)) \cdot [X_2] = 8 \Delta_{14^1 15^3} + 16\Delta_{14^2 15^1 16^1}.
$$
\end{exam}
\begin{exam}\label{ex2}
We consider $m=(2,1,2)$ with $\alpha=(1,3,1)$. Then the canonical exceptional decomposition is $\alpha=(0,1,0)+(1,2,0)+(0,0,1)$ and we have
$$
	n=5,\ \ m=2,\ \ l=2,\ \ r=4,\ \ s=8,\ \ t=7,\ \ \text{and } w=5.
$$ 
With $x_1,\ldots,x_4$ denoting the Chern roots of $\UU^\vee \to \Gr_4(k^{16})$ and $y$ the 1\textsuperscript{st} Chern class of the line bundle $\QQ_0$ on $\Gr_7(k^8)$, we have $\Zn(\Psi) = (x_1+y)^2\ldots(x_4+y)^2$. Like in the previous example, we deduce that
\begin{align*}
	\pi_{1,*}\Zn(\Psi) &= 2 \Delta_{1^1}, \\
	[X_2] &= \Delta_{11^4}, \text{ and thus} \\
	(\pi_{1,*}\Zn(\Psi)) \cdot [X_2] &= 2 \Delta_{11^3 12^1}.
\end{align*}
It is even possible to compute the intersection in this case (we've tried this in the previous example, too -- it's a mess).
Every point of $I^{\alpha}$ corresponds to a commutative diagram

\[
\begin{xy}\xymatrix@R30pt@C29pt{0\ar[r]&S_{\alpha_2}^4\ar@{-->}[rd]^{f_p}\ar[r]^{i_1}\ar[d]^{f_1}&S_{\delta}\ar[d]^{f_2}\ar[r]^{\pi_1}&M_{\hat\alpha}\ar[r]\ar[d]^{f_3}&0\\0\ar[r]&S_{\alpha_1}^7\ar[r]^{i_2}&S_{\alpha_1}^8\ar[r]^{\pi_2}&S_{\alpha_1}\ar[r]&0.}
\end{xy}
\]
The coefficient quiver of $S_{\alpha_1}$ only consists of one vertex which corresponds to a basis element denoted by $b$. We denote the $8$ copies of the basis element $b$ by $b_i$ for $i=1,\ldots,8$. A coefficient quiver of $S_{\alpha_2}$ is given by 
$$\begin{xy}\xymatrix@R12pt@C8pt{&&\bullet\\&v_1\ar[ru]^{\rho^1_1}&&v_2.\ar[lu]_{\rho^2_1}}\end{xy}$$
Then we have 
$$V=\Hom(S_{\alpha_2},S_{\alpha_1}^8)=\langle g_1^i,g_2^i\mid i=1,\ldots,8\rangle$$
where $g_l^i$ is induced by $g_1^i(v_j)=\delta_{1,j}\delta_{i,k}b_k$ and $g_2^i(v_j)=\delta_{2,j}\delta_{i,k}b_k.$
A coefficient quiver of $S_\delta$ where $\delta=(5,10,1)$ is given by
$$\begin{xy}\xymatrix@R12pt@C8pt{&&&&&\bullet&&&&\bullet&\\&&&&w_7\ar[ru]^{\rho^2_1}&&w_8\ar[lu]_{\rho^1_1}&&w_9\ar[ru]^{\rho^2_1}&&w_{10}\ar[lu]_{\rho^1_1}\\&&&&&&&\bullet\ar[lu]_{\rho^2_3}\ar[d]^{\rho_2}\ar[ru]_{\rho^2_3}\ar[rrd]^{\rho_3^1}\ar[lld]_{\rho_3^1}&&&&&&&&\\&&&w_3\ar[rd]_{\rho^1_1}&&w_4\ar[ld]^{\rho^2_1}&&\bullet&&w_5\ar[rd]_{\rho^1_1}&&w_6.\ar[ld]^{\rho^2_1}&\\&&&&\bullet&&w_1\ar[ru]^{\rho^1_1}&&w_2\ar[lu]_{\rho^2_1}&&\bullet&}
\end{xy}$$
 Then $f_2:S_\delta\to S_{\alpha_1}^8$ is defined by
\begin{eqnarray*}f_2(w_1)=b_1,\,f_2(w_2)=b_2,\,f_2(w_3)=b_3,\,f_2(w_4-w_5)=b_4,\\f_2(w_6)=b_5,\,f_2(w_7)=b_6,\,f_2(w_8-w_9)=b_7,\,f_2(w_{10})=b_8.\end{eqnarray*}
Then it is straightforward that 
$$W=\langle g_1^1+g_2^2,g_1^3+g_2^4,g_1^4-g_2^5,g_1^7+g_2^6,g_1^8-g_2^7\rangle.$$
Having chosen a basis of $W$, we identify $\Gr_4(W)$ with $\P(W^*) \cong \P^4$. Under this identification, the closed subset $X_1 \cap X_2 = \{ U \in \Gr_4(W) \mid \dim \pr(U) \leq 7 \}$ corresponds to
$$
	\{ U \in \P(W^*) \mid \dim \Big( \left( A^T \left(\begin{smallmatrix} E_8 \\ 0 \end{smallmatrix}\right) \right)^{-1}U \cap \left( A^T \left(\begin{smallmatrix} 0 \\ E_8 \end{smallmatrix}\right) \right)^{-1}U \Big) \geq 1 \}.
$$
In the above context, $A$ is the transition matrix of the embedding $W \to V$ (with respect to the basis $\{g_j^i\}$). Its transpose is given by
$$
	A^T = \bordermatrix{ 
			& g_1^1 & g_1^2 & g_1^3 & g_1^4 & g_1^5 & g_1^6 & g_1^7 & g_1^8 & g_2^1 & g_2^2 & g_2^3 & g_2^4 & g_2^5 & g_2^6 & g_2^7 & g_2^8 \cr
			& 1 &   &   &   &   &   &   &   &   & 1 &   &   &   &   &   &   \cr
			&   &   & 1 &   &   &   &   &   &   &   &   & 1 &   &   &   &   \cr
			&   &   &   &  1&   &   &   &   &   &   &   &   & -1&   &   &   \cr
			&   &   &   &   &   &   & 1 &   &   &   &   &   &   & 1 &   &   \cr
			&   &   &   &   &   &   &   & 1 &   &   &   &   &   &   & -1&   \cr
	}.
$$
Set $A_{(1)} = A^T \left(\begin{smallmatrix} E_8 \\ 0 \end{smallmatrix}\right)$ and $A_{(2)} = A^T \left(\begin{smallmatrix} 0 \\ E_8 \end{smallmatrix}\right)$. Let $U = [x_1: \ldots : x_5] \in \P(W^*) = \P^4$. With
\begin{align*}
	u_{(1)} &= x_1e^1 + x_2e^3 + x_3e^4 + x_4e^7 + x_5e^8 \\
	u_{(2)} &= x_1e^2 + x_2e^4 - x_3e^5 + x_4e^6 - x_5e^7,
\end{align*}
where $e^1,\ldots,e^8$ is the dual basis of $V_0^*$ to the basis $g^1,\ldots,g^8$, we obtain that
$$
	A_{(1)}^{-1}U \cap A_{(2)}^{-1}U = \big( \langle e^2,e^5,e^6 \rangle + \langle u_{(1)} \rangle \big) \cap \big( \langle e^1,e^3,e^8 \rangle + \langle u_{(2)} \rangle \big).
$$
The condition for this space to be non-zero can be seen to be equivalent to the requirement $x_2x_4 + x_3x_5 = 0$. Therefore, the intersection $I^\alpha = X_1 \cap X_2$ identifies with
$$
	\{ [x_1 : \ldots : x_5] \in \P^4 \mid x_2x_4 + x_3x_5 = 0 \},
$$
which is an irreducible closed subset of $\P^4$ of codimension one.

In this case $I^{\alpha}$ describes the isomorphism classes of indecomposable representations with $\dim\Hom(S_{\alpha_1},M_{\alpha})=1$, see also Corollary \ref{Ringelrefl}. This can be seen as follows: for every indecomposable representation $M_\alpha\in R_\alpha(Q(\un m))$ we have that $1\leq\dim\Hom(S_{\alpha_1},M_{\alpha})\leq 2$. If $\dim\Hom(S_{\alpha_1},M_{\alpha})=1$, there exists a short exact sequence
\[\ses{S_{\alpha_1}}{M_\alpha}{M_{\hat\alpha}}\]
where $M_{\hat\alpha}\in {S_{\alpha_1}}^\perp$ and $\Ext(M_{\hat\alpha},S_{\alpha_1})=\Hom(M_{\hat\alpha},S_{\alpha_1})=k$. 
In particular, $M_{\hat\alpha}$ is indecomposable because $M_\alpha$ is indecomposable and, moreover, it can be written as the quotient of its minimal projective resolution
\[\ses{S_{\alpha_2}^4}{S_\delta}{M_{\hat\alpha}}.\]
This sequence can be completed to a commutative diagram of the desired form and shows $M_{\hat\alpha}\in I^{\alpha}$.

Thus assume that $\dim\Hom(S_{\alpha_1},M_{\alpha})=2$. Then we have $\dim\Hom(N,S_{\alpha_1})=1$ and thus $N_{\rho_3^i}=0$ where $N$ is the quotient of the respective short exact sequence and $\udim N=(1,1,1)$. Thus a basis of $\Ext(N,S_{\alpha_1})$ is induced by the arrows $\rho_3^i$ and we see that we have $M_{\hat\alpha}\notin S_{\alpha_1}^{\perp}$ where $M_{\hat\alpha}$ is given as the middle term of
\[\ses{S_{\alpha_1}}{M_{\hat\alpha}}{N}.\]

\end{exam}

\begin{exam}Let $m=(2,1,2)$ and $\alpha=(1,14,8)$ which is a real root. Then we have 
$$\alpha=(0,2,1)^3+(0,1,0)^6+(1,2,5)$$
$$
	n=8,\ \ m=11,\ \ l=2,\ \ r=2,\ \ s=5,\ \ t=1,\ \ \text{and } w=5.
$$ 
Let $x_1,x_2$ be the Chern roots of $\UU^\vee$ on $\Gr_2(k^{10})$ and $y_1,\ldots,y_4$ be the Chern roots of $\QQ_0$ which lives on $\Gr_1(k^5) = \P^4$. We get $\Zn(\Psi) = \prod_{i=1}^2 \prod_{j=1}^4 (x_i+y_j)^2$ and this yields $\pi_{1,*}\Zn(\Psi) = \smash{( 8m_{1^3}(x_1,x_2) + 2 m_{1^12^1}(x_1,x_2) )^4}$. As $m_{1^3}(x_1,x_2) = 0$, we obtain
$$
	\pi_{1,*}\Zn(\Psi) = 16 m_{1^12^1}(x_1,x_2)^4 = 16 e_{1^12^1}(x_1,x_2)^4 = 16e_{1^42^4}(x_1,x_2).
$$
This can be expressed as a linear combination $\pi_{1,*}\Zn(\Psi) = 16\sum_\lambda K_{\lambda',1^42^4} \Delta_\lambda$. On the other hand, we have $[X_2] = \Delta_{2^2}$. As we are forced to stay in a $2 \times 8$ box, the only term that ``survives'' in the product $(\pi_{1,*}\Zn(\Psi)) \cdot [X_2]$ is the summand attached to $\lambda = 6^2$. Thus, we only need to compute the Kostka number $K_{2^6,1^42^4}$. In other words, we need to fill the boxes
$$
	\yng(2,2,2,2,2,2)
$$
with the integers $1,\ldots,8$ -- the numbers $1,\ldots,4$ occurring twice and the others once -- in such a way that the entries are strictly increasing along both columns and weakly increasing along every row. The following entries are already fixed by this requirement:
$$
	\young(11,22,33,44,5{\ },{\ }8)\, .
$$
We thus see that $K_{2^6,1^42^4} = 2$, whence
$$
	(\pi_{1,*}\Zn(\Psi))\cdot [X_2] = 32 \Delta_{8^2} = 32 [\mathrm{pt}].
$$
Note that we have $\delta=(1,10,5)$ and a coefficient quiver of $S_\delta$ is obtained from the one of $S_\delta$ in Example \ref{ex2} when turning around all arrows. We are left with the question which embedding $S_{\alpha_2}^2\hookrightarrow S_\delta$ gives rise to a surjection of the quotient $M_{\hat\alpha}$ onto $S_{\alpha_1}^4$. Denote by $b_1$ and $b_2$ the two basis elements of $(S_{\alpha_2})_{q_2}$. Then it is straightforward to check that this embedding is induced by $i_1(b_1)=w_8-w_9$ and $i_1(b_2)=w_4-w_5$. Then a coefficient quiver of $M_{\hat\alpha}$ is given when merging the corresponding vertices of the coefficient quiver. 

\end{exam}

\begin{exam}
We consider $Q(2,1,2)$ and the exceptional sequence
\[(\alpha_1,\alpha_2,\alpha_3)=((1,2,0),(2,3,0),(0,0,1)).\]
We consider the dimension vectors $\beta(d)=\alpha_1^{d}+\alpha_2^3+\alpha_3^1$ and $\gamma(e)=\alpha_1^{e}+\alpha_2^4+\alpha_3^1$. Then we have that $\beta(d)$ is a root if and only if $d\leq 3$ and $\gamma(e)$ is a root if and only if $e\leq 2$. Since $(1,3)$ and $(1,4)$ satisfy the glueing conditions, by Theorem \ref{glueing1} we can construct a $15$-parameter family of indecomposable representations of dimension $\beta(3)+\gamma(2)=(19,31,2)$ by glueing representations from $I^{\beta(3)}$ and $I^{\gamma(2)}$.

But note that we cannot construct the real root representation of dimension $(20,33,2)=\alpha_1^6+\alpha_2^7+\alpha_3^2=\beta(d)+\gamma(e)$ where $d+e=6$ by glueing. This is because there is no decomposition $d+e=6$ such that both $\beta(d)$ and $\gamma(e)$ are roots.

\end{exam}


\end{document}